\documentclass[reqno,11pt]{article}

\usepackage{a4wide}
\usepackage{color}
\usepackage{mathrsfs}
\usepackage{mathtools}
\usepackage{amsmath}
\usepackage{amsthm}
\usepackage{amssymb}
\usepackage{bbm}
\usepackage{tikz-cd}
\usepackage{esint}

\numberwithin{equation}{section}
\usepackage[colorlinks,citecolor=green,linkcolor=red]{hyperref}

\usepackage[latin1]{inputenc}

\newcommand{\N}{\mathbb{N}}
\newcommand{\R}{\mathbb{R}}
\newcommand{\B}{\mathbb{B}}
\newcommand{\sfd}{{\sf d}}
\renewcommand{\d}{{\mathrm d}}
\newcommand{\X}{{\rm X}}
\newcommand{\Y}{{\rm Y}}

\newcommand{\1}{\mathbbm 1}

\newcommand{\eps}{\varepsilon}

\newcommand{\fr}{\penalty-20\null\hfill\(\blacksquare\)}

\newtheorem{theorem}{Theorem}[section]
\newtheorem{thmdef}[theorem]{Theorem/Definition}

\newtheorem{lemma}[theorem]{Lemma}
\newtheorem{proposition}[theorem]{Proposition}
\newtheorem{definition}[theorem]{Definition}
\newtheorem{example}[theorem]{Example}
\newtheorem{remark}[theorem]{Remark}

\title{Duals and pullbacks of normed modules}

\author{Nicola Gigli \thanks{SISSA, Via Bonomea 265, 34136 Trieste, Italy. \textit{Email:} {\sf ngigli@sissa.it}}
\and Danka Lu\v{c}i\'{c} \thanks{Universit\`{a} di Pisa, Dipartimento di Matematica, Largo Bruno Pontecorvo 5, 56127 Pisa, Italy.
\textit{Email:} {\sf danka.lucic@dm.unipi.it}}
\and Enrico Pasqualetto \thanks{Scuola Normale Superiore, Piazza dei Cavalieri, 7, 56126 Pisa, Italy.
\textit{Email:} {\sf enrico.pasqualetto@sns.it}}}
\begin{document}
\date{\today}
\maketitle
\begin{abstract}
We give a general description of the dual of the pullback of a normed module. Ours is the natural generalization to the context of modules of the well-known fact that the dual of the Lebesgue-Bochner space $L^p([0,1],B)$ consists - quite roughly said - of $L^q$ maps from $[0,1]$ to the dual $B'$ of $B$ equipped with the weak$^*$ topology. 

In order to state our result, we study various fiberwise descriptions of a normed module that are of independent interest.
\end{abstract}

\bigskip


\tableofcontents
\section{Introduction}

With the aim of developing a differential calculus on metric measure spaces, the first named author introduced in \cite{Gigli14}
the class of \emph{\(L^p\)-normed \(L^\infty\)-modules}. These are Banach spaces endowed with a module structure (over the
ring of \(L^\infty\)-functions) and whose norm is induced by a \emph{pointwise norm} via integration. The prototypical example
is the space of \(p\)-integrable vector fields on a Riemannian manifold, or more generally -- as no continuous structure
is involved -- the \(L^p\)-sections of a measurable Banach bundle over a measure space.
\medskip

The theory of \(L^p\)-normed \(L^\infty\)-modules incorporates the one of Banach spaces -- since the latter can be regarded as normed
modules over a Dirac measure -- and as such it possesses a rich functional-analytic structure. The same theory also wants to be an adaptation of standard tools from differential geometry to more singular settings. As such, several classical constructions from both functional analysis and differential geometry (such as duals and pullbacks) have their natural counterpart in the setting of normed modules and turn out to be very useful in order to
develop an effective vector calculus. For instance, the \emph{tangent module} of a metric measure space (playing the role of the
space of \(2\)-integrable vector fields) is constructed as the module dual of the \emph{cotangent module} (whose elements can be
regarded as the \(2\)-integrable \(1\)-forms); cf.\ with \cite[Sections 2.2 and 2.3]{Gigli14}.

When dealing with this sort of objects, in several circumstances one encounters \emph{duals} and \emph{pullbacks} and it is interesting to know how these operations combine. To give concrete examples, this happens when studying:
\begin{itemize}
\item The \emph{differential} of a map of bounded deformation between two metric measure spaces, introduced in \cite[Proposition 2.4.6]{Gigli14}
(see also \cite[Section 4.2]{GPS18}).
\item The \emph{velocity} of a test plan, introduced in \cite[Theorem 2.3.18]{Gigli14} (see also \cite[Theorem 1.21]{Pasqualetto22}).
\end{itemize}
In both cases, the dual of the pullback of the cotangent module appears: it is important to characterize it and to understand
when it can be identified with the pullback of the tangent module. 

It is worth stressing that Lebesgue--Bochner spaces can be regarded as pullback modules: we have already noticed that a Banach space \(B\) is a normed module over \((\{{\sf o}\},\delta_{\sf o})\).
Moreover, assuming \(\mu_\X(\X)=1\) for simplicity, the Lebesgue--Bochner space \(L^p(\mu_\X;B)\) can be viewed as a normed module over
\(\big(\X\times\{{\sf o}\},\mu_\X\otimes\delta_{\sf o}\big)\cong(\X,\mu_\X)\) and can be identified with the pullback \(p^*B\) of \(B\) under
the unique map \(p\colon\X\times\{{\sf o}\}\to\{{\sf o}\}\). 

This means that understanding the structure of the dual of the pullback of a module means, in particular, understanding the structure of the dual of a Lebesgue--Bochner space and since this latter theory is by now well understood, it can be used as a guideline to study the former. 

Our starting observation is then the fact that the dual \(L^p(\mu_\X;B)'\) of  \(L^p(\mu_\X;B)\) can be regarded -- shortly and roughly said -- as the space of weakly$^*$-measurable $L^q(\mu_\X)$ maps from $\X$ to $B'$: the main result of this manuscript is to generalize this observation to the case of modules and in order to do so we shall need to understand in which sense a module over $\X$ can be seen as space of sections of a suitable Banach bundle over $\X$ itself.

Concerning this latter problem, we are going to provide two different notions of `measurable Banach bundle': a `strong' one, mainly suited  for `primal' modules and conceptually close to the definition of measurability encountered in Lebesgue-Bochner spaces, and a `weak' one, that better describes `dual' modules and conceptually closer to the concept of measurability encountered in Pettis' theory of integration. Let us give some more details:
\begin{itemize}
\item[-] A \emph{strong} Banach bundle is given by a collection $V_\star$ of Banach spaces indexed by points in $\X$ and a collection of `test sections' that are, axiomatically, declared `measurable'. Then a generic section is declared `strongly measurable' provided it can be approximated by these test sections. We remark that this kind of notion is very close to the one of \emph{measurable Banach bundle} studied by Gutman and collaborators (see e.g.\ \cite{Gut93}, \cite{GutKop00} and references therein), although our axiomatization is at times different. Specifically, we do not impose the collection of test sections to be a vector space, but rather a `vector space up to negligible sets', in a sense (see requirement $(a)$ in Definition \ref{def:strbb} for the rigorous formulation): this additional freedom allows us to describe a module as space of strongly measurable sections even without the Axiom of Choice, see below.
\item[-] A \emph{weak} Banach bundle is given by a collection $V'_\star$ of dual Banach spaces indexed by points in $\X$ and a collection of `test vectors', i.e.\ of sections of the collection $V_\star$ of pre-duals. Then a section of  $V'_\star$ is declared `weakly measurable' provided its coupling with any test vector results in a measurable function.
\end{itemize}
As mentioned, we want to use these concepts to describe modules as spaces of sections and beside the theoretical framework provided by the concepts above, we will also need a way of selecting `measurable representatives' of a given a.e.\ defined section. Here as familiar analogy one can think of the problem of assigning to a given (equivalence class up to $\mu_\X$-a.e.\ equality of a) function in $L^p(\mu_\X)$ a measurable representative in  a linear way. We shall study two ways of doing so: these are conceptually analogous, but technically rather different.
\begin{itemize}
\item[1)] On general $\sigma$-finite spaces we shall rely on the theory of \emph{liftings} (\`{a} la von Neumann), see Section \ref{se:gencase}, following a  line of investigation  initiated by the second and the third named authors, together with S.\ Di Marino, in \cite{DMLP21}.
\item[2)] If the underlying measure space is not only $\sigma$-finite, but also separable (a property that always holds if the $\sigma$-algebra is that of Borel sets in a Polish space), then a more constructive procedure is possible, ultimately based on Doob's martingale convergence theorem. An advantage of this approach is that  it does not rely on any use of the Axiom of Choice, beside Countable Dependent. See Section \ref{se:sep}.
\end{itemize}

With these concepts at hand, it will be possible to characterize the dual of the pullback of a given module as space of weakly measurable sections, in a manner strongly resembling the simpler case of the dual of a Lebesgue--Bochner space, see Theorems \ref{thm:pbd1} and  \ref{thm:dpb2}.

\bigskip

The appendices are devoted to further studies about the dual of the pullback of a module. Read in the more familiar context of dual of a Lebesgue--Bochner space, the results are the following:
\begin{itemize}
\item[A)] The dual \(L^p(\mu_\X;B)'\) of \(L^p(\mu_\X;B)\) is isomorphic to the space of `local' maps $T$ from $B$ to $L^q(\mu_\X)$, i.e.\ those linear and continuous maps for which there is $g\in L^q(\mu_\X)^+$ such that $|T(v)|\leq \|v\|_B\,g$ holds $\mu_\X$-a.e.. Here the least (in the $\mu_\X$-a.e.\ sense) such $g$ is called `pointwise norm' of $T$ and denoted by $|T|$, and the space of such local maps is endowed with the $L^q(\mu_\X)$-norm of the pointwise norm. 

The isomorphism is built as follows. Let $L\in L^p(\mu_\X;B)'$ and notice that for $v\in\B$ the map $f\mapsto L(fv)$ sends $L^p(\mu_\X)$ to $\R$ in a linear and continuous way. Hence it is a function in $L^q(\mu_\X)$.
\item[B)] It is well known that $L^q(\mu_\X;B')$ canonically and isometrically embeds in  \(L^p(\mu_\X;B)'\). Here we point out that the image of the embedding is sequentially weakly$^*$ dense in \(L^p(\mu_\X;B)'\).
\end{itemize}
We conclude the Introduction by pointing out that concepts very similar to that of a normed module previously appeared in the literature.
For example, let us only mention the notions of a \emph{randomly normed space} \cite{HLR91} or of a \emph{random normed module}
\cite{Guo-2011}.

\bigskip

\noindent{\bf  Acknowledgements.} We want to thank Proff.\ J.\ Voigt and J.\ Wengenroth for having pointed out to us the lack of weak$^*$ completeness of Banach duals discussed in Remark \ref{re:sorpresa}.

The second named author is supported by the PRIN 2017 `\emph{Gradient flows, Optimal Transport and Metric Measure Structures}'
and the PRIN 2017 `\emph{Variational Methods for Stationary and Evolution Problems with Singularities and Interfaces}', both
funded by the Italian Ministry of Research and University. The third named author is supported by the Balzan project led by Luigi Ambrosio. 

\section{Building modules from sections}

\subsection{Strong Banach bundles}\label{se:sbb}

An \emph{enhanced measurable space} is a triple  \((\X,\Sigma,\mathcal N)\) with $\Sigma$ being a $\sigma$-algebra on the set $\X$ and \(\mathcal N\subset \Sigma\) a \(\sigma\)-ideal (i.e.\ it contains $\varnothing$, is stable by countable unions and $A\subset B$, $B\in\mathcal N$ and $A\in\Sigma$ implies $A\in\mathcal N$).

Elements of $\mathcal N$ are thought of   as the `negligible' subsets in $(\X,\Sigma)$. The typical example of such ideal is the collection $\mathcal N_\mu$ of $\mu$-negligible sets for any given measure $\mu$ on  $(\X,\Sigma)$.

Let  \(V_\star=\{V_x\}_{x\in\X}\) be a collection of Banach spaces indexed by points in $\X$.
Then we define the space of \emph{sections} of \(V_\star\) as
\[
\mathscr S(V_\star)\coloneqq\bigg\{v\colon\X\to\bigsqcup_{x\in\X}V_x\;\bigg|\;v(x)\in V_x,\text{ for every }x\in\X\bigg\}=\prod_{x\in\X}V_x.
\]
Notice that $\mathscr S(V_\star)$ is a vector space, the operations being defined pointwise.

To build a module out of such collection we need to speak about measurable sections and to this aim we need to assume additional structure. In this section we shall assume the existence of a suitable class of sections that are axiomatically taken as measurable and that are used to test measurability of other ones. 

Notice that in the definition below,  the requirements $(i),(ii)$ correspond to being `Borel measurable' and `essentially separably valued' for Lebesgue-Bochner spaces.
\begin{definition}[Strong Banach bundle]\label{def:strbb} A strong Banach bundle is given by: an enhanced measurable  space \((\X,\Sigma,\mathcal N)\), a  collection \(V_\star=\{V_x\}_{x\in\X}\) of Banach spaces and a subset ${\rm Test}(V_\star)$ of  $\mathscr S(V_\star)$ of `test sections' such that:
\begin{itemize}
\item[a)] For every $w_1,w_2\in {\rm Test}(V_\star)$ and $\alpha_1,\alpha_2\in\R$ there are $w\in{\rm Test}(V_\star)$ and $N\in\mathcal N$  such that $\alpha_1w_1(x)+\alpha_2w_2(x)=w(x)$ for every $x\in\X\setminus N$,
\item[b)] $\X\ni x\mapsto\|w(x)\|_{V_x} $ is measurable for any $w\in {\rm Test}(V_\star)$.
\end{itemize} 
We shall denote by $\mathscr S_{\sf str}(V_\star)$ the collection of strongly measurable sections of such  bundle, defined as the set of those $v\in\mathscr S(V_\star)$ such that:
\begin{itemize}
\item[i)] for every $w\in {\rm Test}(V_\star)$ there is $N\in\mathcal N$  such that the map $\X\setminus N\ni x\mapsto \|v(x)-w(x)\|_{V_x}$ is measurable,
\item[ii)] there is a countable collection $(w_n)\subset  {\rm Test}(V_\star)$ and  $N\in\mathcal N$  such that for every $x\in\X\setminus N$ the vector $v(x)$ belongs to the closure of $\{w_n(x)\}$ in $V_x$.
\end{itemize}
\end{definition}
The notation $\mathscr S_{\sf str}(V_\star)$ does not make evident the dependence of this object upon the chosen space ${\rm Test}(V_\star)$, but as this will always be clear from the context, no confusion should occur. 

The following are easily established:
\begin{itemize}
\item[i)] Test sections are strongly measurable.
\item[ii)] For any $v\in \mathscr S_{\sf str}(V_\star)$ there is $N\in\mathcal N$ such that the map $\X\setminus N\ni x\mapsto \|v(x)\|_{V_x}$ is measurable. Indeed, for $(w_n)\subset {\rm Test}(V_\star)$ as in $(ii)$, we can easily find $(A_{n}^k)\subset\Sigma$ and $N\in\mathcal N$ such that the measurable functions $\1_{\X\setminus N}\sum_n \1_{A^k_n}\|v(\cdot)-w_n(\cdot)\|_{V_\cdot}$ go to 0  as $k\to\infty$ for every $x\in\X$. It follows that $\1_{\X\setminus N}\sum_n \1_{A^k_n}\|w_n(\cdot)\|_{V_\cdot}$ go to $\1_{X\setminus N}\|v(\cdot)\|_{V_\cdot}$ pointwise as $k\to\infty$ and the conclusion.
\item[iii)] For $v_1,v_2\in \mathscr S_{\sf str}(V_\star)$ and $\alpha_1,\alpha_2\in\R$ there are $v\in  \mathscr S_{\sf str}(V_\star)$ and $N\in\mathcal N$ such that $\alpha_1v_1(x)+\alpha_2v_2(x)=v(x)$ for any $x\in\X\setminus N$. This follows from the analogous property of  ${\rm Test}(V_\star)$  and an approximation argument like the one just used.
\item[iv)] The class $\mathscr S_{\sf str}(V_\star)$ is stable by sequential pointwise convergence and product with simple functions. In particular, it is stable by product with measurable functions.
\end{itemize}
We now define the equivalence relation $\sim_{{\sf str} }$ on $\mathscr S_{\sf str}(V_\star)$  as: 
\[
v\sim_{\sf str} w\qquad\Leftrightarrow\qquad \|v(x)-w(x)\|_{V_x}=0\quad\forall x\in\X\setminus N,\ \text{ for some $N\in\mathcal N$}.
\]
We also define the $L^\infty(\Sigma,\mathcal N)$-norm as
\[
\|v\|_{L^\infty(\Sigma,\mathcal N)}:=\inf_{N\in\mathcal N}\sup_{x\in \X\setminus N}\|v(x)\|_{V_x}.
\]
Notice that this quantity passes to the quotient and that the subspace $L^\infty(V_\star;\Sigma,\mathcal N)$ of $\mathscr S_{\sf str}(V_\star)/\sim_{\sf str}$ made of (equivalence classes of) sections with finite norm is  a Banach space (and in fact an $L^\infty(\Sigma,\mathcal N)$-normed module in the sense of Definition \ref{def:normed_mod}).

If the $\sigma$-ideal is given by the $\mu$-negligible sets for some measure $\mu$ on $(\X,\Sigma)$, we can introduce $L^p$-sections, $p\in[1,\infty]$, as follows. Fix such $p$ and define
\[
L^p_{\sf str}(V_\star;\mu):=\{v\in \mathscr S_{\sf str}(V_\star)\ :\ \|v(\cdot)\|_{V_\cdot}\in L^p(\mu) \}\ /\ \sim_{\sf str}.
\]
From the above discussion it is clear that the definition is well posed and that $L^p_{\sf str}(V_\star;\mu)$ is an $L^p(\mu)$-normed $L^\infty(\mu)$-module when equipped with its natural pointwise operations (properly passed to the quotient).

\begin{remark}{\rm
Fix a separable \(\sigma\)-finite measure space \((\X,\Sigma,\mu)\). Then one can check that there is a one-to-one correspondence
(up to considering suitable isomorphism classes) between:
\begin{itemize}
\item[i)] Separable strong Banach bundles over \((\X,\Sigma,\mathcal N_\mu)\). Here \emph{separable} means that there is a countable subset
\(\mathcal C\) of \({\rm Test}(V_\star)\) such that for any \(v\in\mathscr S_{\sf str}(V_\star)\) there exist \((w_n)\subset\mathcal C\)
and \(N\in\mathcal N_\mu\) such that \(v(x)\) belongs to the closure of \(\{w_n(x)\}\) in \(V_x\) for every \(x\in\X\setminus N\).
\item[ii)] \emph{Separable measurable Banach \(\mathbb B\)-bundles} on \((\X,\Sigma,\mu)\), in the sense of \cite[Definition 4.1]{DMLP21}.
\end{itemize}
On the one hand, given a separable Banach \(\mathbb B\)-bundle \(\mathbf E\) on \((\X,\Sigma,\mu)\) and a family \(\mathcal C\subset\bar\Gamma(\mathbf E)\)
as in \cite[Proposition 4.4 ii)]{DMLP21}, one obtains a separable strong Banach bundle on \((\X,\Sigma,\mathcal N_\mu)\) by setting
\(V_\star\coloneqq\mathbf E(\star)\) and by choosing as \({\rm Test}(V_\star)\) the \(\R\)-linear span of \(\mathcal C\). On the other hand,
suppose that a separable strong Banach bundle \((V_\star,{\rm Test}(V_\star))\) over \((\X,\Sigma,\mathcal N_\mu)\) is given. By slightly adapting
the arguments in \cite[Remark 4.14]{DMLP21}, one can associate to the strong Banach bundle a \emph{measurable collection of separable Banach spaces}
(in the sense of \cite[Definition 4.5]{DMLP21}), which in turn induces a separable Banach \(\mathbb B\)-bundle \(\mathbf E\) over \((\X,\Sigma,\mu)\)
by \cite[Corollary 4.7]{DMLP21}.

Consequently, it follows from \cite[Theorem 4.15]{DMLP21} that for any \(p\in[1,\infty]\) the section functor
\((V_\star,{\rm Test}(V_\star))\mapsto L^p_{\sf str}(V_\star;\mu)\) is a one-to-one correspondence between (equivalence classes of)
separable strong Banach bundles on \((\X,\Sigma,\mathcal N_\mu)\) and separable \(L^p(\mu)\)-normed modules.
\fr}\end{remark}

\subsection{Weak Banach bundles}\label{se:wbb}

We are now going to consider a different kind of Banach bundles, where measurability is tested via linear functionals. As before, we are given an enhanced measurable space $(\X,\Sigma,\mathcal N)$ and a collection \(V_\star=\{V_x\}_{x\in\X}\) of Banach spaces indexed by points in $\X$. We shall denote by  \(V'_\star=\{V'_x\}_{x\in\X}\) the collection of the topological duals of the spaces, and by $\langle\cdot,\cdot\rangle$ the duality pairing between any
of the spaces $V_x$ and its topological dual $V'_x$.

Notice that the requirement \eqref{eq:weakmeas} below corresponds to the weak measurability requirement appearing in the theory of Pettis integration.
\begin{definition}[Weak Banach bundle]\label{def:wbb} A weak Banach  bundle is given by: an enhanced measurable space \((\X,\Sigma,\mathcal N)\), a collection \(V_\star=\{V_x\}_{x\in\X}\) of Banach spaces and a subset ${\rm Test}(V_\star)$ of  $\mathscr S(V_\star)$ of `test vectors' such that:
\begin{itemize}
\item[a)]  For every $v_1,v_2\in {\rm Test}(V_\star)$ and $\alpha_1,\alpha_2\in\R$ there are $v\in{\rm Test}(V_\star)$ and $N\in\mathcal N$  such that $\alpha_1v_1(x)+\alpha_2v_2(x)=v(x)$ for every $x\in\X\setminus N$.
\end{itemize}
We shall denote by $\mathscr S_{\sf weak}(V'_\star)\subset \mathscr S(V'_\star)$ the collection of weakly measurable sections   defined as the set of those $w\in\mathscr S(V'_\star)$ such that: for every test vector $v\in {\rm Test}(V_\star)$ there is $N\in\mathcal N$ such that 
\begin{equation}
\label{eq:weakmeas}
\X\setminus N \ni x\quad\mapsto\quad \langle w,v\rangle(x)\coloneqq\langle w(x),v(x)\rangle\in\R\text{ is measurable}.
\end{equation}
\end{definition}
As before, the notation $\mathscr S_{\sf weak}(V'_\star)$ hides the dependence of such object on the chosen test vectors, but given that these will always be clear from the context, no confusion should occur. We notice also that the same definition would work (and in fact would be more general) if we chose to test measurability via elements in the dual space, rather than the predual: our choice is motivated by the fact that in all our applications we are dealing with test vectors as above.

It is clear from the definition that $\mathscr S_{\sf weak}(V'_\star)$  is a vector space, stable by sequential pointwise convergence and by product with measurable functions. Also, denoting by  $\widetilde{\rm Test}(V_\star)\subset \mathscr S(V_\star)$ the collection of $v$'s of the form $v=\sum_{i=1}^nf_iv_i$ for $n\in\N$, $f_i:\X\to\R$ measurable and $(v_i)\subset {\rm Test}(V_\star)$ (the operations being defined pointwise), it is clear that: for every $w\in\mathscr S_{\sf weak}(V'_\star) $ and $v\in \widetilde{\rm Test}(V_\star)$ there is $N\in\mathcal N$ such that 
\[
\X\setminus N\ni x\quad\mapsto\quad \langle w,v\rangle(x) \text{ is measurable}.
\]
We now define an equivalence relation $\sim_{\sf weak}$ on $\mathscr S_{\sf weak}(V'_\star) $ as
\[
\begin{split}
w_1\sim_{\sf weak} w_2\qquad\Leftrightarrow\qquad&\forall v\in {\rm Test}(V_\star)\ \exists N\in\mathcal N\text{ such that } \langle w_1-w_2,v\rangle(x)=0\quad\forall x\in \X\setminus N  
\\
\Leftrightarrow\qquad &\forall v\in \widetilde{\rm Test}(V_\star)\ \exists N\in\mathcal N\text{ such that } \langle w_1-w_2,v\rangle(x)=0\quad\forall x\in \X\setminus N  .
\end{split}
\]
To elements of $\mathscr S_{\sf weak}(V'_\star) /\sim_{\sf weak}$ we would like to associate  a `pointwise norm'. It is unclear to us whether this is always possible, but there are at least two relevant cases where it is so. 

The first is when $\mathcal N=\{\varnothing\}$: in this case the norm can be defined pointwise:
\[
|w|(x):=\sup_{\genfrac{}{}{0pt}{2}{v\in \widetilde{\rm Test}(V_\star):}{\|v(\cdot)\|_{V_\cdot}\leq 1}}\langle w,v\rangle(x) \qquad\forall x\in\X
\]
(notice that $x\mapsto |w|(x)$ is not necessarily measurable).

The second is when $\mathcal N$ coincides with the collection of $\mu$-negligible sets for some $\sigma$-finite measure $\mu$ on $(\X,\Sigma)$. In this case we put
\begin{equation}
\label{eq:defnw}
|w|:=\underset{\genfrac{}{}{0pt}{2}{v\in \widetilde{\rm Test}(V_\star):}{\|v(\cdot)\|_{V_\cdot}\leq 1}}{\rm ess\,sup} \langle w,v\rangle,
\end{equation}
notice that this well defines up to $\mu$-a.e.\ equality a measurable $[0,\infty]$-valued function, that we call \emph{pointwise norm} of $w$ and, rather trivially, satisfies the bound
\begin{equation}
\label{eq:dualnorm}
|\langle w,v\rangle|\leq |w|\|v(\cdot)\|_{V_\cdot}\quad\mu-a.e.\qquad\forall w\in \mathscr S_{\sf weak}(V'_\star),\ v\in \widetilde{\rm Test}(V_\star).
\end{equation}
It is readily verified that $\sim_{\sf weak}$ also respects the vector space structure and that the product by measurable functions induces a product with $L^\infty(\mu)$-functions on the quotient. Then for $p\in[1,\infty]$ we can consider the collection  $L^p_{\sf weak}(V'_\star;\mu)$ of $L^p(\mu)$-weakly  integrable elements of  $\mathscr S_{\sf weak}(V'_\star)$  as
\[
L^p_{\sf weak}(V'_\star;\mu):=\{w\in \mathscr S_{\sf weak}(V'_\star)\ :\  |w|\in L^p(\mu) \}\ /\ \sim_{\sf weak}.
\]
On this space the map $w\mapsto \||w|\|_{L^p(\mu)}$ is a norm and what just said trivially ensures that, if we knew that $\big(L^p_{\sf weak}(V'_\star;\mu),\||\cdot|\|_{L^p(\mu)}\big)$ is complete, we could conclude that $L^p_{\sf weak}(V'_\star;\mu)$ is an $L^p(\mu)$-normed $L^\infty(\mu)$-module (i.e.\ all the properties are trivially verified except completeness). In fact, completeness is in place, but its proof is not immediate: we will detail it only in specific cases of interest, but the argument generalizes, see Theorems \ref{prop:chardual}, \ref{thm:pbd1}, \ref{thm:dual2} and \ref{thm:dpb2}.

\section{The theory for separable measure spaces}\label{se:sep}

\subsection{Linear choice of  representatives}

A {separable} measure space  $(\X,\Sigma,\mu)$ is a measure space with $\mu\geq0$ and such that $\Sigma$ equipped with the semidistance  $\sfd(A,B):=\mu(A\Delta B)$ (that can take the value $+\infty$ and be 0 on different sets) is separable as a metric space, more precisely:  there are $(A_n)\subset\Sigma$ such that for any $B\in\Sigma$ with $\mu(B)<\infty$ we have $\inf_n\mu(A_n\Delta B)=0$. Such a collection of sets $(A_n)$ will be called a dense sequence in  $(\X,\Sigma,\mu)$.

In this chapter we shall be concerned with separable and $\sigma$-finite measure spaces $(\X,\Sigma,\mu)$.

\bigskip

We are interested in picking Borel representatives of $L^1$ functions that linearly depend on the (equivalence class of the) given function in $L^1(\mu)$. On separable and $\sigma$-finite spaces this is possible, via the following construction, at least if we give ourselves the freedom of defining the representative only on some measurable set with negligible complement. Below we shall denote by $L^1(\mu)$ the usual Banach space of $\mu$-equivalence classes of integrable functions and by $\mathcal L^1(\mu)$ the collection of real valued $\Sigma$-measurable maps that are $\mu$-integrable. For given $f:\X\to\R$ $\Sigma$-measurable, by $\pi_\mu(f)$ we denote its equivalence class up to $\mu$-a.e.\ equality.

Fix a  separable and $\sigma$-finite space $(\X,\Sigma,\mu)$ and a dense sequence $(A_n)$. Then  for every $k\in\N$  let $\Sigma_k$ be the (finite) $\sigma$-algebra generated by $A_1,\ldots,A_k$ and let  $\mathcal P_k$ be the finest partition of $\X$ made of sets in $\Sigma_k$. Notice that $\mathcal P_{k+1}$ is a refinement of $\mathcal P_k$.

In \cite{BG22} it has been proved the following result (that in turn is  a direct consequence of Doob's martingale convergence theorem):
\begin{theorem}\label{thm:partition_Doob}
Let \((\X,\Sigma,\mu)\) be a separable and $\sigma$-finite  space, $(A_n)\subset\Sigma$ a dense sequence and $\mathcal P:=\{\mathcal P_k:k\in\N\}$ the collection of partitions $\mathcal P_k=\{E_k^n\}_n$ defined as above.

 For any \(k\in\N\), let us define the linear, \(1\)-Lipschitz operator \(P_k\colon L^1(\mu)\to\mathcal L^1(\mu)\) as
\[
P_k(f)\coloneqq\sum_{n:\,\mu(E_k^n)>0}\bigg(\fint_{E_k^n}f\,\d\mu\bigg)\1_{E_k^n},\quad\text{ for every }f\in L^1(\mu).
\]
For \(f\in L^1(\mu)\), we define   \({\rm Leb}(f)\in\Sigma \) and the measurable function \({\rm Rep}(f)\colon\X\to\R\) as
\[\begin{split}
{\rm Leb}(f)&\coloneqq\Big\{x\in\X\;\Big|\;\exists\lim_{k\to\infty}P_k(f)(x)\in\R\Big\},\\
{\rm Rep}(f)&\coloneqq\left\{\begin{array}{ll}
\lim_{k\to\infty}P_k(f)(x),\\
0,
\end{array}\quad\begin{array}{ll}
\text{ if }x\in{\rm Leb}(f),\\
\text{ if }x\in\X\setminus{\rm Leb}(f),
\end{array}\right.
\end{split}\]
respectively. Then it holds that \(\mu\big(\X\setminus{\rm Leb}(f)\big)=0\) and \(\pi_\mu\big({\rm Rep}(f)\big)=f\) for every \(f\in L^1(\mu)\).
Moreover, for any \(f\in L^1(\mu)\) it holds that \(P_k(f)\to f\) strongly in \(L^1(\mu)\) as \(k\to\infty\).
\end{theorem}
A simple consequence of the above is:
\begin{proposition}\label{prop:properties_Leb} With the same assumptions and notation as  in Theorem \ref{thm:partition_Doob} above, we have:
\begin{itemize}
\item[\(\rm i)\)] Let \(f,g\in L^1(\mu)\) be such that \(f\leq g\) in the \(\mu\)-a.e.\ sense. Then
\begin{equation}\label{eq:ineq_BR}
{\rm Rep}(f)(x)\leq{\rm Rep}(g)(x),\quad\text{ for every }x\in{\rm Leb}(f)\cap{\rm Leb}(g).
\end{equation}
\item[\(\rm ii)\)] Let \(f,g\in L^1(\mu)\) and \(\alpha,\beta\in\R\) be given. Then \({\rm Leb}(f)\cap{\rm Leb}(g)\subset{\rm Leb}(\alpha f+\beta g)\) and
\[
{\rm Rep}(\alpha f+\beta g)(x)=\alpha\,{\rm Rep}(f)(x)+\beta\,{\rm Rep}(g)(x),\quad\text{ for every }x\in{\rm Leb}(f)\cap{\rm Leb}(g).
\]
In particular, given any \(x\in\X\), the collection \(\big\{f\in L^1(\mu)\,:\,x\in{\rm Leb}(f)\big\}\) is a linear subspace of \(L^1(\mu)\) 
and \(\big\{f\in L^1(\mu)\,:\,x\in{\rm Leb}(f)\big\}\ni f\mapsto{\rm Rep}(f)(x)\in\R\) is a linear function.
\end{itemize}
\end{proposition}
\begin{proof}
To prove i), observe that \(P_k(f)\leq P_k(g)\) for every \(k\in\N\) whenever \(f\leq g\) holds \(\mu\)-a.e..
To prove ii), just notice that \(P_k(\alpha f+\beta g)=\alpha P_k(f)+\beta P_k(g)\) holds for every \(k\in\N\).
\end{proof}
\subsection{Modules as spaces of sections}\label{se:mss}
Standing assumptions of this section are:
\begin{itemize}
\item[-]  $(\X,\Sigma_\X,\mu_\X)$ and  $(\Y,\Sigma_\Y,\mu_\Y)$ are  separable and $\sigma$-finite measure spaces,
\item[-] $\varphi:\Y\to\X$ is a measurable map with $\varphi_*\mu_\Y=\mu_\X$,
\item[-] $p,q\in(1,\infty)$ are with $\frac1p+\frac1q=1$,
\item[-] $\mathscr M$ is an $L^p(\mu_\X)$-normed module.
\end{itemize}

\subsubsection{The starting module \texorpdfstring{$\mathscr M$}{M}}\label{se:startm}

We assume the reader familiar with the concept of $L^p$-normed $L^\infty$-module (or briefly $L^p$-normed module) as developed in \cite{Gigli14}.  Our goal in this section is to show that our given module $\mathscr M$ is isomorphic to the space of $L^p$-integrable sections of some strong Banach bundle.  We shall do so by constructing such Banach bundle starting from $\mathscr M$ itself and to this aim we shall start from  the construction in Theorem \ref{thm:partition_Doob}. We thus fix a sequence of partitions  \(\mathcal P\)  of $\X$ as in such theorem, then we consider the following definition:
\begin{definition}\label{def:lebp}
Let \((\X,\Sigma_\X,\mu_\X)\) be a separable and $\sigma$-finite measure space and \(\mathcal P\) as in Theorem \ref{thm:partition_Doob}. Let \(p\in(1,\infty)\) be a given exponent. Then
to any function \(f\in L^p(\mu_\X)^+\) we associate the  set \({\rm Leb}_p(f)\in\Sigma_\X\) and the measurable function \({\rm Rep}_p(f)\colon\X\to[0,+\infty)\), defined as
\[
{\rm Leb}_p(f)\coloneqq{\rm Leb}(f^p),\qquad{\rm Rep}_p(f)(x)\coloneqq{\rm Rep}(f^p)(x)^{1/p},\quad\text{ for every }x\in\X,
\]
respectively.
\end{definition}
Clearly, ${\rm Rep}_p$ is not  linear as ${\rm Rep}$, but key properties, useful in particular when studying duality relations, are retained, as we are now going to see.

Let \(\mathcal P=(\mathcal P_k)_{k\in\N}\) be our partitions, say \(\mathcal P_k=(E_k^n)_n\).  We  set
\[
E_k(x)\coloneqq E_k^n,\quad\text{ for every }k,n\text{ and }x\in E_k^n.
\]
\begin{lemma}\label{lem:BR_p} With the  notation and assumptions as in Section \ref{se:mss} the following holds:
\begin{itemize}
\item[\(\rm i)\)] Let \(f,g\in L^p(\mu_\X)\) and \(\alpha,\beta\in\R\) be given. Then
\begin{equation}\label{eq:ineq_BR_p}
{\rm Rep}_p(|\alpha f+\beta g|)(x)\leq|\alpha|\,{\rm Rep}_p(|f|)(x)+|\beta|\,{\rm Rep}_p(|g|)(x),
\end{equation}
for every \(x\in{\rm Leb}_p(|f|)\cap{\rm Leb}_p(|g|)\cap{\rm Leb}_p(|\alpha f+\beta g|)\).
\item[\(\rm ii)\)] Let \(f\in L^p(\mu_\X)^+\) and \(g\in L^q(\mu_\X)^+\) be given. Then
\[
{\rm Rep}(fg)(x)\leq{\rm Rep}_p(f)(x)\,{\rm Rep}_q(g)(x),\quad\text{ for every }x\in{\rm Leb}_p(f)\cap{\rm Leb}_q(g)\cap{\rm Leb}(fg).
\]
\end{itemize}
\end{lemma}
\begin{proof}
To prove i), observe that for any \(k\in\N\) we have that
\[\begin{split}
P_k\big(|\alpha f+\beta g|^p\big)(x)^{1/p}&=\frac{\|\alpha f+\beta g\|_{L^p(\mu_\X|_{E_k(x)})}}{\mu_\X(E_k(x))^{1/p}}
\leq|\alpha|\frac{\|f\|_{L^p(\mu_\X|_{E_k(x)})}}{\mu_\X(E_k(x))^{1/p}}+|\beta|\frac{\|g\|_{L^p(\mu_\X|_{E_k(x)})}}{\mu_\X(E_k(x))^{1/p}}\\
&=|\alpha|P_k(|f|^p)(x)^{1/p}+|\beta|P_k(|g|^p)(x)^{1/p},
\end{split}\]
whence \eqref{eq:ineq_BR_p} follows by letting \(k\to\infty\). To prove ii), notice that for any \(k\in\N\) we have that
\[\begin{split}
P_k(fg)(x)&=\fint_{E_k(x)}fg\,\d\mu_\X\leq\bigg(\fint_{E_k(x)}f^p\,\d\mu_\X\bigg)^{1/p}\bigg(\fint_{E_k(x)}g^q\,\d\mu_\X\bigg)^{1/q}\\
&=P_k(f^p)(x)^{1/p}P_k(g^q)(x)^{1/q}
\end{split}\]
by H\"{o}lder's inequality. Then  \({\rm Rep}(fg)(x)\leq{\rm Rep}_p(f)(x){\rm Rep}_q(g)(x)\) follows letting \(k\to\infty\).
\end{proof}
We turn to the study of $\mathscr M$.  For any \(x\in\X\) we define the subset \(\mathscr M_x\subset\mathscr M\) as
\[
\mathscr M_x\coloneqq\big\{v\in\mathscr M\;\big|\;x\in{\rm Leb}_p(|v|)\big\}.
\]
Notice that \(\mathscr M_x\) is a \(1\)-homogeneous subset of \(\mathscr M\) (\emph{i.e.}, \(\alpha v\in\mathscr M_x\) for every \(v\in\mathscr M_x\)
and \(\alpha\in\R\)), but needs not be a linear subspace of \(\mathscr M\). We denote by \(\tilde{\mathscr M}_x\) the linear span of
\(\mathscr M_x\), namely,
\[
\tilde{\mathscr M}_x\coloneqq\bigg\{\sum_{i=1}^n v_i\;\bigg|\;n\in\N,\,(v_i)_{i=1}^n\subset\mathscr M_x\bigg\}.
\]
Then \(\tilde{\mathscr M}_x\) is a linear subspace of \(\mathscr M\). We introduce a seminorm \(\|\cdot\|_x\) on \(\tilde{\mathscr M}_x\) as follows:
\[
\|v\|_x\coloneqq\inf\bigg\{\sum_{i=1}^n{\rm Rep}_p(|v_i|)(x)\;\bigg|\;n\in\N,\,(v_i)_{i=1}^n\subset\mathscr M_x,\,v=\sum_{i=1}^n v_i\bigg\},
\quad\text{ for every }v\in\tilde{\mathscr M}_x.
\]
Notice that
\begin{equation}
\label{eq:consist_norm_fiber}
\|v\|_x={\rm Rep}_p(|v|)(x),\quad\text{ for every }v\in\mathscr M_x,\ \forall x\in\X.
\end{equation}
Indeed,  \(\leq\) is trivial while to prove the converse one we pick any \(v\in\tilde{\mathscr M}_x\) and \(v_1,\ldots,v_n\in\mathscr M_x\)
such that \(v=v_1+\ldots+v_n\). Then \({\rm Rep}_p(|v|)(x)\leq\sum_{i=1}^n{\rm Rep}_p(|v_i|)(x)\) by Lemma \ref{lem:BR_p} i). The inequality
\({\rm Rep}_p(|v|)(x)\leq\|v\|_x\) follows by  the arbitrariness of \(v_1,\ldots,v_n\).

We introduce an equivalence relation \(\sim_x\) on \(\tilde{\mathscr M}_x\), by declaring that two elements \(v,w\in\tilde{\mathscr M}_x\)
satisfy \(v\sim_x w\) if and only if \(\|v-w\|_x=0\). Then the quotient space \(\tilde{\mathscr M}_x/\sim_x\) inherits a norm, which we still
denote by \(\|\cdot\|_x\). We call \(q_x\colon\tilde{\mathscr M}_x\to\tilde{\mathscr M}_x/\sim_x\) the projection on the quotient.
Moreover, we denote by \((\hat{\mathscr M}_x,\hat\iota_x)\) the  completion of \(\big(\tilde{\mathscr M}_x/\sim_x,\|\cdot\|_x\big)\).
Then \(\hat{\mathscr M}_x\) is a Banach space, whose norm we denote again by \(\|\cdot\|_x\). Notice that
\(\iota_x\coloneqq\hat\iota_x\circ q_x\colon\tilde{\mathscr M}_x\to\hat{\mathscr M}_x\) is a linear operator which preserves the (semi)norm.

Clearly, we now have a collection $\hat{\mathscr M}_\star$ of Banach spaces indexed by points in $\X$. To get a strong Banach bundle we need to define a collection of test sections as in Definition \ref{def:strbb}.

Given any element \(v\in\mathscr M\), we define its `measurable representative' 
${\rm Rep}(v)\in \mathscr S(\hat{\mathscr M}_\star)$ as
\[
{\rm Rep}(v)(x)\coloneqq\left\{\begin{array}{ll}
\iota_x(v),\\
0\in\hat{\mathscr M}_x,
\end{array}\quad\begin{array}{ll}
\text{ if }x\in{\rm Leb}_p(|v|),\\
\text{ if }x\in\X\setminus{\rm Leb}_p(|v|).
\end{array}\right.
\]
From the linearity of $\iota_x$ it follows that for any $v_1,v_2\in\mathscr M$ and $\alpha_1,\alpha_2\in\R$, putting $v:=\alpha_1v_1+\alpha_2v_2$ we have
\begin{equation}
\label{eq:perlin}
{\rm Rep}(v)(x)=\alpha_1{\rm Rep}(v_1)(x)+\alpha_2{\rm Rep}(v_2)(x)\qquad\forall x\in {\rm Leb}_p(|v_1|)\cap {\rm Leb}_p(|v_2|)\cap{\rm Leb}_p(|v|).
\end{equation}
This identity ensures that the collection $\{{\rm Rep}(v):v\in\mathscr M\}$ satisfies property $(a)$ in Definition \ref{def:strbb}, while \eqref{eq:consist_norm_fiber} shows that $(b)$ holds (with $\mathcal N:=\mathcal N_{\mu_\X}$). Thus $\{{\rm Rep}(v):v\in\mathscr M\}$ is an admissible family of test sections of  the collection $\hat{\mathscr M}_\star$.

Denoting by $[\cdot]$ the equivalence class of a strongly measurable section w.r.t.\ $\sim_{\sf str}$, it trivially follows from  \eqref{eq:consist_norm_fiber}  that
\begin{equation}
\label{eq:stessanorma}
\|[{\rm Rep}(v)]\|_\cdot= |v|,\qquad\mu_\X-a.e.
\end{equation}
and thus in particular that $\|[{\rm Rep}(v)]\|_\cdot\in L^p(\mu_\X)$.

We then have:
\begin{proposition}\label{prop:realis_mod}
With the notation and assumptions as in Section \ref{se:mss}, the map 
\[
\mathscr M\ni v\mapsto [{\rm Rep}(v)]\in L^p_{\sf str}(\hat{\mathscr M}_\star;\mu_\X)
\] is an isomorphism of $L^p(\mu)$-normed  modules.
\end{proposition}
\begin{proof}
The fact that the given map preserves multiplication with simple functions is obvious. Then taking also \eqref{eq:stessanorma} into account we see that it is continuous, preserves the pointwise norm and also the product with $L^\infty$ functions. It remains to prove surjectivity. In turn, this is a quite direct consequence of the requirement $(ii)$ in the definition of strongly measurable section given in Definition \ref{def:strbb} and an approximation argument.
\end{proof}

\subsubsection{The pullback}
Proposition \ref{prop:realis_mod} above gives a representation of $\mathscr M$ as space of sections: starting from it we want to provide an analogue description of  the pullback $\varphi^*\mathscr M$.

Let us first recall the definition of the pullback $\varphi^*\mathscr M$ of $\mathscr M$ (see \cite{Gigli14}):
\begin{thmdef}[Pullback of an $L^p$-normed module]\label{thm:defpb} With the notation and assumptions as in Section \ref{se:mss}, the following holds. There exists a unique couple \((\varphi^*\mathscr M,\varphi^*)\), called
the {pullback} of \(\mathscr M\) under \(\varphi\), where \(\varphi^*\mathscr M\) is an $L^p(\mu_\Y)$-normed module and
\(\varphi^*\colon\mathscr M\to\varphi^*\mathscr M\) is linear and such that:
\begin{itemize}
\item[\(\rm i)\)] The identity \(|\varphi^*v|=|v|\circ\varphi\) holds for every \(v\in\mathscr M\).
\item[\(\rm ii)\)] The space $\{\varphi^*v:v\in \mathscr M\}$ generates \(\varphi^*\mathscr M\) in the sense of $L^\infty(\mu_\Y)$-modules.
\end{itemize}
Uniqueness is up to unique isomorphism: given any \((\mathscr N,T)\) with the same properties, there exists a unique 
isomorphism \(\Phi\colon\varphi^*\mathscr M\to\mathscr N\) such that \(T=\Phi\circ\varphi^*\).
\end{thmdef}

In the previous section we built, starting from an $L^p(\mu_\X)$-normed module $\mathscr M$, a family $\hat{\mathscr M}_\star$ of Banach spaces indexed by points in $\X$. Now we consider on $\Y$ the collection of Banach spaces $\hat{\mathscr M}_{\varphi(\star)}$, i.e.\ the Banach space corresponding to $y\in\Y$ is $\hat{\mathscr M}_{\varphi(y)}$. Also, as test sections consider the collection of maps of the form $y\mapsto {\rm Rep}(v)(\varphi(y))$ as $v$ varies in $\mathscr M$ (i.e.\ the pullbacks via $\varphi$ of the test sections considered before). 

From the assumption $\varphi_*\mu_\Y=\mu_\X$  and \eqref{eq:consist_norm_fiber} and \eqref{eq:perlin} it is immediate to verify that this collection of sections satisfies the requirements $(a),(b)$ in Definition \ref{def:strbb}, hence what we just described is a strong Banach bundle on $(\Y,\Sigma_\Y,\mathcal N_{\mu_\Y})$.

We then have the following characterization of the pullback module. Notice that it is in line with the classical interpretation of pullback bundle in differential geometry, where the fibre at $y$ of the pullback bundle is the same as the fibre at $\varphi(y)$ of the original bundle.
\begin{proposition}
With the notation and assumptions as in Section \ref{se:mss} and just described we have
\[
\big(L^p_{\sf str}(\hat{\mathscr M}_{\varphi(\star)};\mu_\Y),{\rm Pb}\big) \cong \big(\varphi^*\mathscr M,\varphi^*\big) ,
\] 
where the pullback map ${\rm Pb}:\mathscr M\to L^p_{\sf str}(\hat{\mathscr M}_{\varphi(\star)};\mu_\Y)$ is the one sending $v$ to the $\sim_{\sf str}$-equivalence class of  $y\mapsto {\rm Rep}(v)(\varphi(y))$.
\end{proposition}
\begin{proof} By \eqref{eq:stessanorma} it follows that
\[
\|{\rm Pb}(v)\|_{\hat{\mathscr M}_{\varphi(x)}}=|v|(\varphi(x)),\qquad\mu_\X-a.e.\ x\in\X
\]
while the very definition of space of  $L^p$ sections of a strong Banach bundle ensures that (the image after quotienting by $\sim_{\sf str}$ of) the test sections generate the whole $L^p$ module (this follows by the requirement $(ii)$ of strongly measurable section).

The conclusion follows from the uniqueness part of Theorem \ref{thm:defpb}.
\end{proof}

\subsubsection{The dual} We denote by \(\hat{\mathscr M}'_\star\) the collection
\(\big\{(\hat{\mathscr M}_x)'\big\}_{x\in\X}\) of dual Banach spaces: we are now going to see that  $\mathscr M^*$ is isomorphic to a space of sections of such collection. Since duality is involved, it is perhaps not surprising that the concept of weak Banach bundle comes into play.

As space of test vectors we pick the collection of maps of the form $x\mapsto {\rm Rep}(v)(x)\in \hat{\mathscr M}_x$ as $v$ varies in $\mathscr M$. By \eqref{eq:perlin} it is clear that this collection satisfies the requirement $(a)$ in Definition \ref{def:wbb} (with $\mathcal N=\mathcal N_{\mu_\X}$). In other words, the space $\X$, the collection  \(\hat{\mathscr M}'_\star\)  of Banach spaces and the set $\{{\rm Rep}(v):v\in\mathscr M\}$ of test vectors form a weak Banach bundle.

We recall from \cite{Gigli14}:
\begin{definition}[Dual of an $L^p$-normed module]\label{def:dualm}
The dual $\mathscr M^*$ of the $L^p(\mu_\X)$-normed module $\mathscr M$ is the space of $L^\infty(\mu_\X)$-linear and continuous maps from $\mathscr M$ to $L^1(\mu_\X)$. It is equipped with the operator norm, with the obvious multiplication by $L^\infty(\mu_\X)$-functions and with the pointwise norm
\[
|L|:=\underset{\genfrac{}{}{0pt}{2}{v\in\mathscr M:}{|v|\leq 1\ \mu_\X-a.e.}}{\rm ess\,sup} L(v).
\]
With this structure, $\mathscr M^*$ is an $L^q(\mu_\X)$-normed module, $\frac1p+\frac1q=1$.
\end{definition}
We then have:
\begin{theorem}\label{prop:chardual} With the above notation and assumptions, we have that $L^q_{\sf weak}(\hat{\mathscr M}'_\star;\mu_\X)$ is an $L^q(\mu_\X)$-normed module and that the map 
\[
\begin{array}{rccc}
{\sf I}:&L^q_{\sf weak}(\hat{\mathscr M}'_\star;\mu_\X)\qquad&\to&\qquad\mathscr M^*\\
&\omega_\cdot\qquad&\mapsto&\quad\big(v\quad\mapsto\quad [\langle \omega_\cdot,{\rm Rep}(v)(\cdot)\rangle]\big)
\end{array}
\]
is an isomorphism of modules, where here $[f]$ denotes the equivalence class up to $\mu_\X$-a.e.\ equality of the $\mu_\X$-measurable function $f$.
\end{theorem}
\begin{proof} If $\omega_\cdot$ is weakly measurable, then by the choice  of test vectors we see that $x\mapsto \langle \omega_x,{\rm Rep}(v)(x)\rangle$ is $\mu_\X$-measurable. Then taking also into account the bound \eqref{eq:dualnorm} we see that ${\sf I}$ is well defined and $|{\sf I}(\omega_\cdot)|\leq |\omega|$ $\mu_\X$-a.e.. Also, ${\sf I}$ is clearly $L^\infty$-linear, thus it only remains to prove that it is surjective and that  $|{\sf I}(\omega_\cdot)|\geq |\omega|$ $\mu_\X$-a.e.\ (notice that this and the completeness of $\mathscr M^*$ give the completeness of $L^q_{\sf weak}(\hat{\mathscr M}'_\star;\mu_\X)$ and thus ensure that this is an $L^q(\mu_\X)$-normed module).

To this aim, let \(L\in\mathscr M^*\) be fixed. Given any \(x\in{\rm Leb}_q(|L|)\), we define \(\tilde{\mathcal V}_x\subset\tilde{\mathscr M}_x\) as
\[
\tilde{\mathcal V}_x\coloneqq\big\{v\in\tilde{\mathscr M}_x\;\big|\;x\in{\rm Leb}(\langle L,v\rangle)\big\}
\]
and \(\mathcal V_x\coloneqq\iota_x(\tilde{\mathcal V}_x)\subset\hat{\mathscr M}_x\). Notice that \(\mathcal V_x\) is a linear subspace of
\(\hat{\mathscr M}_x\) by Proposition \ref{prop:properties_Leb} ii). Clearly, the function \(L_x\colon\mathcal V_x\to\R\), which we define as
\[
L_x(v)\coloneqq{\rm Rep}(\langle L,\tilde v\rangle)(x),\quad\text{ for every }v\in\mathcal V_x
\text{ and }\tilde v\in\tilde{\mathcal V}_x\text{ with }v=\iota_x(\tilde v),
\]
is well-posed and linear. We aim to show that \(L_x\) is \({\rm Rep}_q(|L|)(x)\)-Lipschitz when its domain \(\mathcal V_x\) is endowed with the norm
\(\|\cdot\|_x\) induced from \(\hat{\mathscr M}_x\). Letting \(\mathcal P=(\mathcal P_k)_{k\in\N}\) and \(\mathcal P_k=(E^n_k)_n\) be the partitions used to define ${\rm Rep}$, we set
\(E_k(x)\coloneqq E^n_k\) for every \(k,n\) and \(x\in E^n_k\). Then for any \(x\in{\rm Leb}_q(|L|)\), \(k\in\N\),
\(\tilde v\in\tilde{\mathcal V}_x\), and \((v_i)_{i=1}^n\subset\mathscr M_x\) with \(\tilde v=\sum_{i=1}^n v_i\), we can estimate
\[\begin{split}
\bigg|\fint_{E_k(x)}\langle L,\tilde v\rangle\,\d\mu_\X\bigg|&\leq\fint_{E_k(x)}\big|\langle L,\tilde v\rangle\big|\,\d\mu_\X
\leq\bigg(\fint_{E_k(x)}|L|^q\,\d\mu_\X\bigg)^{1/q}\bigg(\fint_{E_k(x)}|\tilde v|^p\,\d\mu_\X\bigg)^{1/p}\\
&\leq\bigg(\fint_{E_k(x)}|L|^q\,\d\mu_\X\bigg)^{1/q}\sum_{i=1}^n\bigg(\fint_{E_k(x)}|v_i|^p\,\d\mu_\X\bigg)^{1/p}.
\end{split}\]
Since \(x\in{\rm Leb}_p(|v_i|)\) for all \(i=1,\ldots,n\) and \(x\in{\rm Leb}(\langle L,\tilde v\rangle)\), by letting \(k\to\infty\) we deduce that
\[\begin{split}
|L_x(v)|&=\big|{\rm Rep}(\langle L,\tilde v\rangle)(x)\big|=\lim_{k\to\infty}\bigg|\fint_{E_k(x)}\langle L,\tilde v\rangle\,\d\mu_\X\bigg|\\
&\leq\lim_{k\to\infty}\bigg(\fint_{E_k(x)}|L|^q\,\d\mu_\X\bigg)^{1/q}\sum_{i=1}^n\bigg(\fint_{E_k(x)}|v_i|^p\,\d\mu_\X\bigg)^{1/p}\\
&={\rm Rep}_q(|L|)(x)\sum_{i=1}^n{\rm Rep}_p(|v_i|)(x),
\end{split}\]
where we set \(v\coloneqq\iota_x(\tilde v)\). Thanks to the arbitrariness of \((v_i)_{i=1}^n\), we can conclude that
\(|L_x(v)|\leq{\rm Rep}_q(|L|)(x)\|\tilde v\|_x={\rm Rep}_q(|L|)(x)\|v\|_x\), which shows that \(L_x\) is \({\rm Rep}_q(|L|)(x)\)-Lipschitz.

Hence, by using the Hahn--Banach Theorem we can find an element \(\bar\omega(x)\in(\hat{\mathscr M}_x)'\) such that
\(\bar\omega(x)|_{\mathcal V_x}=L_x\) and \(\|\bar\omega(x)\|_{(\hat{\mathscr M}_x)'}\leq{\rm Rep}_q(|L|)(x)\).
Moreover, for any \(x\in\X\setminus{\rm Leb}_q(|L|)\) we set \(\bar\omega(x)\coloneqq 0\in(\hat{\mathscr M}_x)'\). It is then easy to verify
that \(\bar\omega\in\mathscr S_{\sf weak}\big(\hat{\mathscr M}'_\star)\), that its equivalence class \(\omega\) w.r.t.\ $\sim_{\sf weak}$  
belongs to $L^q_{\sf weak}(\hat{\mathscr M}'_\star;\mu_\X)$, that \(|\omega|\leq|L|\) in the \(\mu_\X\)-a.e.\ sense, and
that \({\sf I}(\omega)=L\).
\end{proof}

\subsubsection{The dual of the pullback}
A combination of the arguments used in these last two paragraphs provide a characterization of the dual of $\varphi^*\mathscr M$ as space of sections of a suitable weak Banach bundle.

We shall make use of the following property of pullback modules, whose proof can be easily obtained with minor modifications of  \cite[Proposition 1.6.3]{Gigli14}:

\begin{proposition}\label{prop:univ_prop_pullback-pol} With the notation and assumptions as in Section \ref{se:mss}, the following holds.

Suppose that \(T\colon\mathscr M\to L^1(\mu_\Y)\) is a linear operator for which there exists a function \(g\in L^q(\mu_\Y)\) such that \(|T(v)|\leq g|v|\circ\varphi\)
holds for every \(v\in\mathscr M\). 

Then there exists a unique linear and continuous map \(\hat T\colon\varphi^*\mathscr M\to L^1(\mu_\Y)\) such that
\[
\hat T(\varphi^*v)=T(v),\quad\text{ for every }v\in\mathscr M.
\]
Moreover, it holds that \(|\hat T(V)|\leq g|V|\) for every \(V\in\varphi^*\mathscr M\).
\end{proposition}
By  ${\hat{\mathscr M}}_{\star}$ we denote the collection of Banach spaces over $\X$ defined in Section \ref{se:startm} and by  ${\hat{\mathscr M}}_{\star}'$ the collection of duals.

We are going to consider the weak Banach bundle on $(\Y,\Sigma_\Y,\mathcal N_{\mu_\Y})$ given by:
\begin{itemize}
\item[-] The collection ${\hat{\mathscr M}}'_{\varphi(\star)}$, i.e.\ the Banach space corresponding to $y\in\Y$ is $({\hat{\mathscr M}}_{\varphi(y)})'$.
\item[-] The set of test vectors $\{{\rm Rep}(v)\circ\varphi:v\in\mathscr M\}\subset\mathscr S({\hat{\mathscr M}}_{\varphi(\star)})$.
\end{itemize}
From \eqref{eq:perlin} it follows  that this is a weak Banach bundle in the sense of Definition \ref{def:wbb}.

We are now going to define a map
\[
\mathcal I:L^q_{\sf weak}(\hat{\mathscr M}'_{\varphi(\star)};\mu_\Y)\qquad\to\qquad(\varphi^*\mathscr M)^*
\]
as follows: for any  $\omega_\cdot\in L^q_{\sf weak}(\hat{\mathscr M}'_{\varphi(\star)};\mu_\Y)$ and chosen a $\sim_{\sf weak}$-representative $\bar\omega_\cdot$ of $\omega_\cdot$, the map 
\[
\mathscr M\ni v\mapsto [\langle\bar\omega_\cdot,{\rm Rep}(v)\circ\varphi \rangle]\in L^1(\mu_\Y)
\] 
(here $[f]$ is the $\mu_\Y$-a.e.\ equivalence class of the $\mu_\Y$-measurable function $f$) is easily seen to be well defined, independent of the chosen representative, linear and - by \eqref{eq:dualnorm} and \eqref{eq:stessanorma} - to satisfy
\[
\big|[\langle\bar\omega_\cdot,{\rm Rep}(v)\circ\varphi \rangle]\big|\leq |\omega_\cdot|\,|v|\circ\varphi\qquad\mu_\Y-a.e..
\]
Hence by Proposition \ref{prop:univ_prop_pullback-pol} there is a unique $L^\infty(\mu_\Y)$-linear and continuous map $\mathcal I(\omega_\cdot)$ from $\varphi^*\mathscr M$ to $L^1(\mu_\Y)$ (i.e.\ an element of $(\varphi^*\mathscr M)^*$) such that
\[
\mathcal I(\omega_\cdot)(\varphi^*v)=[\langle\bar\omega_\cdot,{\rm Rep}(v)\circ\varphi \rangle]\quad\mu_\Y-a.e.\ \qquad\forall v\in\mathscr M
\]
and such map satisfies
\begin{equation}
\label{eq:bpbtr}
|\mathcal I(\omega_\cdot)|\leq  |\omega_\cdot|\qquad\mu_\Y-a.e..
\end{equation}

We then have:
\begin{theorem}\label{thm:pbd1} With the  notation and assumptions as in Section \ref{se:mss} and above, the space $L^q_{\sf weak}(\hat{\mathscr M}'_{\varphi(\star)};\mu_\Y)$ is an $L^q(\mu_\Y)$-normed module and the map $\mathcal I$  is an isomorphism of $L^q(\mu_\Y)$-normed modules.
\end{theorem}
\begin{proof}
We know that $\mathcal I$ is linear and satisfies \eqref{eq:bpbtr}. Since $L^\infty(\mu_\Y)$-linearity is trivial, to conclude it is sufficient to check surjectivity and the bound $|\mathcal I(\omega_\cdot)|\geq  |\omega_\cdot|$ (as in Theorem \ref{prop:chardual} this also gives the completeness of $L^q_{\sf weak}(\hat{\mathscr M}'_{\varphi(\star)};\mu_\Y)$ and thus the fact that it is an $L^q(\mu_\Y)$-normed module).

To this aim, pick a dense sequence of sets in $(\Y,\Sigma_\Y,\mu_\Y)$ and define the corresponding objects  ${\rm Leb}$, ${\rm Rep}$, ${\rm Leb}_q$, ${\rm Rep}_q$ as in Theorem \ref{thm:partition_Doob}  and Definition \ref{def:lebp}. Now  let \(L\in(\varphi^*\mathscr M)^*\)
be fixed. Given any \(y\in{\rm Leb}_q(|L|)\), we define \(\tilde{\mathcal V}_y\subset\tilde{\mathscr M}_{\varphi(y)}\) as
\[
\tilde{\mathcal V}_y\coloneqq\big\{v\in\tilde{\mathscr M}_{\varphi(y)}\;\big|\;y\in{\rm Leb}(\langle L,\varphi^*v\rangle)\big\}
\]
and \(\hat{\mathcal V}_y\coloneqq\iota_{\varphi(y)}(\tilde{\mathcal V}_y)\subset\hat{\mathscr M}_{\varphi(y)} \). Notice that \(\hat{\mathcal V}_y\) is a vector space by Proposition \ref{prop:properties_Leb} ii).

We define the function \(L_y\colon\hat {\mathcal V}_y\to\R\) as
\[
L_y(v)\coloneqq{\rm Rep}(\langle L,\varphi^*\tilde v\rangle)(y),\quad\text{ for every }v\in\hat{\mathcal V}_y
\text{ and }\tilde v\in\tilde{\mathcal V}_y\text{ with }v=\iota_{\varphi(y)}(\tilde v).
\]
Arguing as in the proof of Proposition \ref{prop:chardual}, one can show that the function \(L_y\) is well-defined, linear, continuous,
and satisfying \(\|L_y\|_{(\hat{\mathcal V}_y)'}\leq{\rm Rep}_q(|L|)(y)\). Hence, using the Hahn--Banach Theorem we obtain an element
\(\bar\omega(y)\in(\hat{\mathscr M}_{\varphi(y)})'\) with \(\bar\omega(y)|_{\hat{\mathcal V}_y}=L_y\) and \(\|\bar\omega(y)\|_{(\hat{\mathscr M}_{\varphi(y)})'}
\leq{\rm Rep}_q(|L|)(y)\). Moreover, for any \(y\in\Y\setminus{\rm Leb}_q(|L|)\) we set \(\bar\omega(y)\coloneqq 0\in(\hat{\mathscr M}_{\varphi(y)})'\).

By construction, for any $v\in\mathscr M$ we have \(v\in\hat{\mathcal V}_y\) for \(\mu_\Y\)-a.e.\ \(y\in\Y\), hence the identity
\begin{equation}
\label{eq:idfin}
\langle \bar \omega(y),{\rm Rep}(v)(\varphi(y))\rangle=\langle \bar \omega(y),\iota_{\varphi(y)}(v)\rangle=\langle L_y,\iota_{\varphi(y)}(v)\rangle={\rm Rep}(\langle L,\varphi^* v\rangle)(y)
\end{equation}
is valid for  \(\mu_\Y\)-a.e.\ \(y\in\Y\) and shows that $\bar\omega\in\mathscr S(\hat{\mathscr M}'_{\varphi(\star)})$ is weakly measurable. Letting $\omega$ be the $\sim_{\sf weak}$-equivalence class of $\bar \omega$, the above identity and the definition \eqref{eq:defnw} show that   \(|\omega|\leq|L|\)
in the \(\mu_\Y\)-a.e.\ sense, so that \(|\omega|\in L^q(\mu_\X)\) and thus
\(\omega\in L^q_{\sf weak}(\hat{\mathscr M}_{\varphi(\star)}';\mu_\Y)\). 

Since \eqref{eq:idfin} also shows that $\mathcal I(\omega)=L$, the proof is complete.
\end{proof}
\begin{remark}{\rm
Both in the proof of Theorem \ref{prop:chardual} and of Theorem \ref{thm:pbd1}, the Hahn--Banach Theorem is used. This is due more to our chosen axiomatization of weak Banach bundle than to an actual need (we chose the proposed definition just for simplicity of exposition): a fully satisfactory alternative definition of weak Banach bundle requires the typical section $w$ not to be so that $w(x)$ is a linear continuous functional from $V_x$ to $\R$ but rather that $w(x)$ is a linear continuous functional defined \emph{on some subspace of $V_x$}. These subspaces must be sufficiently big so that \emph{for every test vector $v$ we should have $v(x)$ in the domain of $w(x)$ for every $x\in\X\setminus N$, for some $N\in\mathcal N$} (plus the necessary measurability requirements).

It is immediate to verify that  with this alternative definition the proofs of Theorems \ref{prop:chardual}, \ref{thm:pbd1} work without modifications and without the requirement of any Choice other than Countable Dependent.
\fr}\end{remark}

\section{The general case}\label{se:gencase}

\subsection{Linear choice of representatives}

When dealing with possibly non-separable $\sigma$-finite measure spaces, Theorem \ref{thm:partition_Doob} is not anymore available and to select measurable representatives we have to use  the lifting theory of von Neumann.

Let us recall some terminology and results, referring   to \cite{Fremlin3} and the references therein for a thorough account of this topic.
\begin{definition}[Lifting]
Let \((\X,\Sigma_\X,\mu_\X)\) be a $\sigma$-finite  measure space and $\bar\Sigma_\X$ the completion of $\Sigma_\X$ obtained adding $\mu_\X$-measurable sets. Let \(\Xi\) be a \(\sigma\)-subalgebra of \(\Sigma_\X\).
Then a mapping \(\phi\colon\Xi\to\bar\Sigma_\X\) is said to be a \emph{pre-lifting} of \(\mu_\X\) provided the following hold:
\[\begin{split}
\phi(\varnothing)=\varnothing&,\\
\phi(\X)=\X&,\\
\phi(E\cup F)=\phi(E)\cup\phi(F)&,\quad\text{ for every }E,F\in\Xi,\\
\phi(E\cap F)=\phi(E)\cap\phi(F)&,\quad\text{ for every }E,F\in\Xi,\\
\phi(E)=\phi(F)&,\quad\text{ for every }E,F\in\Xi\text{ with }\mu_\X(E\Delta F)=0,\\
\mu_\X\big(E\Delta\phi(E)\big)=0&,\quad\text{ for every }E\in\Xi.
\end{split}\]
If \(\Xi=\Sigma_\X\), then we say that \(\phi\) is a \emph{lifting} of \(\mu_\X\). Liftings will be typically denoted by \(\ell\).
\end{definition}
\begin{theorem}[Extension of a pre-lifting]\label{thm:ext_pre-lift} 
Let \((\X,\Sigma_\X,\mu_\X)\) be a $\sigma$-finite  measure space and $\bar\Sigma_\X$ the completion of $\Sigma_\X$ obtained adding $\mu_\X$-measurable sets. Let \(\Xi\) be a \(\sigma\)-subalgebra of \(\Sigma_\X\) and \(\phi\colon\Xi\to\bar\Sigma_\X\) a pre-lifting of \(\mu_\X\). 

Then there exists a lifting \(\ell\) of \(\mu_\X\) such that \(\ell|_\Xi=\phi\).
\end{theorem}
By choosing \(\Xi\coloneqq\{\varnothing,\X\}\) and $\Sigma_{\X}=\bar\Sigma_\X$ in Theorem \ref{thm:ext_pre-lift}, one recovers the celebrated \emph{von Neumann--Maharam Theorem},
stating that each (complete, $\sigma$-finite) measure space admits a lifting. More precisely, this theorem is usually stated for complete and finite measure spaces, but the $\sigma$-finite case can easily be obtained by a patching argument.
 
\medskip

Each lifting induces a lifting at the level of bounded measurable functions: given a lifting \(\ell\) of \(\mu_\X\),
there exists a unique linear and continuous operator \(\ell\colon L^\infty(\mu_\X)\to\mathcal L^\infty(\bar\Sigma_\X)\) such that
\[
\ell([\1_E])=\1_{\ell(E)},\quad\text{ for every }E\in\Sigma_\X,
\]
where $[f]$ is the equivalence class in $L^\infty(\mu_\X)$ of $f\in \mathcal L^\infty(\Sigma_\X)$.  Moreover, it is straightforward to check that the following properties are verified:
\[\begin{split}
\|\ell(f)\|_{\mathcal L^\infty(\bar\Sigma_\X)}=\|f\|_{L^\infty(\mu_\X)}&,\quad\text{ for every }f\in L^\infty(\mu_\X),\\
\pi_{\mu_\X}(\ell(f))=f&,\quad\text{ for every }f\in L^\infty(\mu_\X),\\
\ell(fg)=\ell(f)\ell(g)&,\quad\text{ for every }f,g\in L^\infty(\mu_\X),\\
|\ell(f)|=\ell(|f|)&,\quad\text{ for every }f\in L^\infty(\mu_\X),\\
\ell(f)\leq\ell(g)&,\quad\text{ for every }f,g\in L^\infty(\mu_\X)\text{ with }f\leq g\text{ }\mu_\X\text{-a.e..}
\end{split}\]

\subsection{Preliminaries on module theory}

A feature  of the theory of liftings is that it `only' works on the space $L^\infty$, meaning that it has no counterpart for $L^p$ spaces for $p<\infty$. This forces us to work with $L^\infty$-normed $L^\infty$-modules and to tailor  a bit the concept of duality to ensure that the dual of an  $L^\infty$-normed module is still $L^\infty$-normed (rather than $L^1$-normed as it would be from Definition \ref{def:dualm}). This tweaking of the theory is mostly unharmful, as we are going to discuss in Remark \ref{rmk:consist_L_p}.

One thing that we gain from focussing on $L^\infty$-normed modules is that there is  a theory that treat on equal grounds both the modules we are originally interested in and their liftings: it is the theory of $L^\infty(\Sigma_\X,\mathcal N)$-normed modules that has been introduced in  \cite{DMLP21} as a variant of the analogous notion appeared in \cite{Gigli14} and that we are now going to recall.

\medskip

Let us fix an enhanced measurable space $(\X,\Sigma_\X,\mathcal N)$ (i.e.\ $\mathcal N\subset\Sigma_\X$ is  a $\sigma$-ideal). Denote by $\mathcal L^\infty(\Sigma_\X)$ the vector space of real valued and bounded measurable functions. The $\sigma$-ideal $\mathcal N$  induces an equivalence relation
\(\sim_{\mathcal N}\) on \(\mathcal L^\infty(\Sigma_\X)\): given any \(f,g\in\mathcal L^\infty(\Sigma_\X)\), we declare that \(f\sim_{\mathcal N}g\) if
\[
\big\{x\in\X\;\big|\;f(x)\neq g(x)\big\}\in\mathcal N.
\]
We call \(L^\infty(\Sigma_\X,\mathcal N)\) the quotient space \(\mathcal L^\infty(\Sigma_\X)/\sim_{\mathcal N}\), and \(\pi_{\mathcal N}\colon\mathcal L^\infty(\Sigma_\X)\to L^\infty(\Sigma_\X,\mathcal N)\) 
the projection map. The linear space \(L^\infty(\Sigma_\X,\mathcal N)\)  is a Banach space if endowed with the norm
\[
\|f\|_{L^\infty(\Sigma_\X,\mathcal N)}\coloneqq\inf_{N\in\mathcal N}\sup_{\X\setminus N}|\bar f|,\quad\text{ for every }f=\pi_{\mathcal N}(\bar f)\in L^\infty(\Sigma_\X,\mathcal N).
\]
Note that \(L^\infty(\Sigma_\X,\mathcal N)\) is a commutative ring with respect to the usual pointwise operations and that the natural partial ordering in $\mathcal L^\infty(\Sigma_\X)$ passes to the quotient.  
We introduce the shorthand notation
\[
\1_E^{\mathcal N}\coloneqq\pi_{\mathcal N}(\1_E)\in L^\infty(\Sigma_\X,\mathcal N),\quad\text{ for every }E\in\Sigma_\X.
\]
The choice $\mathcal N=\{\varnothing\}$ is possible: in this case we have \(\mathcal L^\infty(\Sigma_\X)=L^\infty(\Sigma_\X,\{\varnothing\})\).

On the other hand, any given measure \(\mu\geq 0\)
on \((\X,\Sigma_\X)\) is associated with the \(\sigma\)-ideal \(\mathcal N_\mu\) of its negligible sets, defined as $\mathcal N_\mu\coloneqq\big\{N\in\Sigma_\X\;\big|\;\mu(N)=0\big\}$.
Then \(L^\infty(\mu)=L^\infty(\Sigma_\X,\mathcal N_\mu)\) as Banach spaces. For brevity, for any $E\in\Sigma_\X$ we shall write
$\1_E^\mu\in L^\infty(\mu)$ instead of $\1_E^{\mathcal N_\mu}$.

We come to the related concept of module.
\begin{definition}[\(L^\infty(\Sigma_\X,\mathcal N)\)-normed module]\label{def:normed_mod}
Let \((\X,\Sigma_\X,\mathcal N)\) be an enhanced measurable space. An \(L^\infty(\Sigma_\X,\mathcal N)\)-normed module is a Banach space  \((\mathscr M,\|\cdot\|_{\mathscr M})\) that is also a module over the ring \(L^\infty(\Sigma_\X,\mathcal N)\) and possesses a map \(|\cdot|\colon\mathscr M\to L^\infty(\Sigma_\X,\mathcal N)\), called \emph{pointwise norm}, with the following properties:
\begin{itemize}
\item[i)] Given any \(v,w\in\mathscr M\) and \(f\in L^\infty(\Sigma_\X,\mathcal N)\), it holds that
\[\begin{split}
|v|&\geq 0,\quad\text{ with equality if and only if }v=0,\\
|v+w|&\leq|v|+|w|,\\
|f\cdot v|&=|f||v|.
\end{split}\]
\item[ii)] Given any partition \((E_n)_{n\in\N}\subset\Sigma_\X\) of $\X$  and \((v_n)_{n\in\N}\subset\mathscr M\) with
\[
\sup_{n\in\N}\big\|\1_{E_n}^{\mathcal N}|v_n|\big\|_{L^\infty(\Sigma_\X,\mathcal N)}<+\infty,
\]there exists an element \(v\in\mathscr M\) such that $\1_{E_n}^{\mathcal N}\cdot v=\1_{E_n}^{\mathcal N}\cdot v_n$, for every $n\in\N$.
\item[iii)] The norm $\|\cdot\|_{\mathscr M}$ can be recovered via the formula
\[
\|v\|_{\mathscr M}\coloneqq\big\||v|\big\|_{L^\infty(\Sigma_\X,\mathcal N)},\quad\text{ for every }v\in\mathscr M.
\]
\end{itemize}
\end{definition}

The property stated in item ii) is called the \emph{glueing property}. The element \(v\) appearing therein is uniquely determined, as one
can readily check, and thus  can be unambiguously denoted by \(\sum_{n\in\N}\1_{E_n}^{\mathcal N}\cdot v_n\in\mathscr M\).
We point out that \(\sum_{n\in\N}\1_{E_n}^{\mathcal N}\cdot v_n\) is only a formal expression, since it is not necessarily the
\(\|\cdot\|_{\mathscr M}\)-limit of \(\sum_{n=1}^N\1_{E_n}^{\mathcal N}\cdot v_n\) as \(N\to\infty\).
\medskip

It will be convenient to denote by \(\mathscr G(\mathcal V)\subset\mathscr M\) the linear space of all those elements of \(\mathscr M\)
that can be obtained by glueing together elements of a given linear space \(\mathcal V\subset\mathscr M\). Namely,
\[
\mathscr G(\mathcal V)\coloneqq\bigg\{\sum_{n\in\N}\1_{E_n}^{\mathcal N}\cdot v_n\;\bigg|\;(E_n)_{n\in\N}\subset\Sigma_\X\text{ partition of }\X,\,
(v_n)_{n\in\N}\subset\mathcal V,\,\sup_{n\in\N}\|\1_{E_n}^{\mathcal N}\cdot v_n\|_{\mathscr M}<\infty\bigg\}.
\]
If $\mathscr G(\mathcal V)$ is dense in $\mathscr M$ we say that $\mathcal V$ \emph{generates} $\mathscr M$.

The theory developed in this part of the work is based on the following result,  proved in  \cite[Theorem 3.5]{DMLP21}:
\begin{theorem}[Lifting of modules]\label{thm:lift_mod}
Let \((\X,\Sigma_\X,\mu_\X)\) be a $\sigma$-finite measure space, \(\mathscr M\) be an \(L^\infty(\mu_\X)\)-normed module, $\bar\Sigma_\X$ the completion of $\Sigma_\X$ and \(\ell:\Sigma_\X\to\bar\Sigma_\X\) a lifting of \(\mu_\X\).

Then there exists a unique couple \((\ell\mathscr M,\ell)\), where \(\ell\mathscr M\) is an \(\mathcal L^\infty(\bar\Sigma_\X)\)-normed module
and \(\ell\colon\mathscr M\to\ell\mathscr M\) is a linear operator, such that the following hold:
\begin{itemize}
\item[\(\rm i)\)] The identity \(|\ell(v)|=\ell(|v|)\) holds for every \(v\in\mathscr M\).
\item[\(\rm ii)\)] The linear space \(\{\ell(v)\,:\,v\in\mathscr M\}\) generates  \(\ell\mathscr M\).
\end{itemize}
Uniqueness is up to unique isomorphism: given any \((\mathscr N,T)\) with the same properties, there exists a unique
\(\mathcal L^\infty(\bar\Sigma_\X)\)-normed module isomorphism \(\Phi\colon\ell\mathscr M\to\mathscr N\) such that \(T=\Phi\circ\ell\).
\end{theorem}
Notice that with the notation of this last theorem we have
\begin{equation}
\label{eq:prodlift}
\ell(fv)=\ell(f)\,\ell(v)\qquad\forall f\in L^\infty(\mu_\X), \ v\in\mathscr M.
\end{equation}
To see this, by linearity and continuity it is sufficient to check that $\ell(\1_E^{\mu_\X}v)=\ell(\1_E^{\mu_\X})\,\ell(v)$ holds for any $E\in\Sigma_\X$ and $v\in\mathscr M$. In turn to prove this we notice that $\ell(\1_E^{\mu_\X})=\1_{\ell(E)}$ and that $\1_{\ell(E)}+\1_{\ell(\X\setminus E)}\equiv 1$ (because liftings preserve unions and intersections and send $\varnothing, \X$ to $\varnothing, \X$, respectively). Then observe that
\[
|\ell(\1_{\X\setminus E}^{\mu_\X})\ell(\1_E^{\mu_\X}v)|=|\ell(\1_{\X\setminus E}^{\mu_\X})|\,|\ell(\1_E^{\mu_\X}v)|=\ell(\1_{\X\setminus E}^{\mu_\X})\,\ell(|\1_E^{\mu_\X}v|)=\ell(\1_{\X\setminus E}^{\mu_\X}\,|\1_E^{\mu_\X}v|)=\ell(0)\equiv 0,
\]
so that $\ell(\1_{\X\setminus E}^{\mu_\X})\ell(\1_E^{\mu_\X}v)=0$  in $\ell\mathscr M$. Since the same holds inverting the roles of $E,\X\setminus E$ we deduce that
\[
\ell(v)=\big(\ell(\1_E^{\mu_\X})+\ell(\1_{\X\setminus E}^{\mu_\X})\big)\big(\ell(\1_E^{\mu_\X}v)+\ell(1_{\X\setminus E}^{\mu_\X}v)\big)=
\ell(\1_E^{\mu_\X})\ell(\1_E^{\mu_\X}v)+\ell(\1_{\X\setminus E}^{\mu_\X})\ell(1_{\X\setminus E}^{\mu_\X}v)
\]
and multiplying this identity first by $\ell(\1_E^{\mu_\X})$ and then by $\ell(\1_{\X\setminus E}^{\mu_\X})$ our claim $\ell(\1_E^{\mu_\X}v)=\ell(\1_E^{\mu_\X})\,\ell(v)$  follows.

Much like quotienting $\mathcal L^\infty(\bar\Sigma_\X)$ by a $\sigma$-ideal $\mathcal N$ we obtain $L^\infty(\bar\Sigma_\X,\mathcal N)$, we can take the quotient of an $\mathcal L^\infty(\bar\Sigma_\X)$-normed module $\bar{\mathscr M}$ and obtain an $L^\infty(\bar\Sigma_\X,\mathcal N)$-normed module.  Specifically,   we can define an equivalence relation \(\sim_{\mathcal N}\) on
\(\bar{\mathscr M}\) by declaring  \(v\sim_{\mathcal N} w\) provided
\(\pi_{\mathcal N}(|v-w|)=0\). Then it is easily seen that  the quotient space
\[
\Pi_{\mathcal N}(\bar{\mathscr M})\coloneqq\bar{\mathscr M}/\sim_{\mathcal N}
\]
inherits the structure of an $L^\infty(\bar\Sigma_\X,\mathcal N)$-normed module. We will denote by
\(\pi_{\mathcal N}\colon\bar{\mathscr M}\to\Pi_{\mathcal N}(\bar{\mathscr M})\) the canonical projection map. If $\mathcal N=\mathcal N_\mu$ is the $\sigma$-ideal of $\mu$-negligible sets for some given measure $\mu$, we shall use the formalism $\sim_\mu$, $\pi_\mu$ and $\Pi_\mu (\bar{\mathscr M})$ for the above.
  
A first example of this construction is in the following result, proven in \cite[Lemma 3.7]{DMLP21}, that shows that if we first lift a module and then quotient it back we return to the original module:
\begin{lemma}\label{lem:quotient_lift}
Let \((\X,\Sigma_\X,\mu_\X)\) be a $\sigma$-finite measure space, $\bar\Sigma_\X$ the completion of $\Sigma_\X$, \(\ell\) a lifting of \(\mu_\X\) and $\mathscr M$ an \(L^\infty(\mu_\X)\)-normed module. 

Then  \(\pi_{\mu_\X}\circ\ell\colon\mathscr M\to\Pi_{\mu_\X}(\ell\mathscr M)\) is an  isomorphism of \(L^\infty(\mu_\X)\)-normed modules.
\end{lemma}
\subsection{Modules as spaces of sections}\label{se:mss2}

Standing assumptions of this section are:
\begin{itemize}
\item[-]  $(\X,\Sigma_\X,\mu_\X)$ and  $(\Y,\Sigma_\Y,\mu_\Y)$ are  $\sigma$-finite measure spaces,
\item[-] $\bar\Sigma_\X$ and $\bar\Sigma_\Y$ are the completion of $\Sigma_\X$ and $\Sigma_\Y$  w.r.t.\ the respective measures,
\item[-] the (at most countable) sets $A_i\in\bar\Sigma_\X$ are the lifted atoms of $\X$ as in Lemma \ref{le:la},
\item[-] $\varphi:\Y\to\X$ is a $(\bar\Sigma_\Y,\bar\Sigma_\X)$-measurable map with $\varphi_*\mu_\Y=\mu_\X$,
\item[-] $\ell_\X$ and $\ell_\Y$ are liftings of $\mu_\X$ and $\mu_\Y$ respectively that are compatible through $\varphi$ (see Definition \ref{def:complif}),
\item[-] $p,q\in(1,\infty)$ are with $\frac1p+\frac1q=1$,
\item[-] $\mathscr M$ is an $L^\infty(\mu_\X)$-normed module and $(\ell_\X\mathscr M,\ell_\X)$ its lifting as in Theorem \ref{thm:lift_mod}.
\end{itemize}

\subsubsection{The starting module \texorpdfstring{$\mathscr M$}{M}}\label{se:basem2} Here we describe how the theory of liftings allows to provide representations of general modules as space of sections in line with the general concepts introduced in Section \ref{se:sbb}.

The idea we will follow is to start from a module $\mathscr M$, to take its lifting $\ell\mathscr M$ and then to define the `fibre at $x$ of $\mathscr M$' as the space $\1_{\{x\}}\cdot\ell\mathscr M$. However, this procedure is not applicable if the singleton $\{x\}$ is not a measurable set; to deal with this possibility, the following simple lemma is useful.
\begin{lemma}[Lifted atoms]\label{le:la} 
Let $(\X,\Sigma_\X,\mu_\X)$ be a $\sigma$-finite measure space,  $\bar\Sigma_\X$ the completion of $\Sigma_\X$ and $\ell:\Sigma_\X\to\bar\Sigma_\X$ a lifting of $\mu_\X$. Let $(A'_i)_{i\in I}\subset\Sigma_\X$, with $I$ at most countable, be (a choice of representatives of) the atoms of $\mu_\X$. Put $A_i:=\ell(A_i')\in\bar\Sigma_\X$.

Then:
\begin{itemize}
\item[i)] For every $x\in \X\setminus \cup_iA_i$ we have $\{x\}\in\bar\Sigma_\X$.
\item[ii)] For every $E\in\Sigma_\X$ and $i\in I$ we either have $\ell(E)=A_i$ or $\ell(E)\cap A_i=\varnothing$. In particular, for any $f\in L^\infty(\mu_\X)$ and $i\in I$ the function $\ell(f)$ is constant on $A_i$.
\end{itemize}
\end{lemma}
\begin{proof}\ \\
\noindent{(i)} Let $R:= \X\setminus \cup_iA_i$, $x\in R$ and let $m:=\inf\mu_\X(B)$, where the infimum is taken among all $B\in\Sigma_\X$ with $x\in B\subset R$. If $m>0$ we could find a minimizing sequence $B_n$, then the set $B:=\cap_nB_n\in\Sigma_\X$ would be an atom disjoint from $R$, which contradicts the construction. Hence $m=0$ and we can find sets $B_n\subset R$ containing $x$ with $\mu_\X(B_n)\leq\tfrac1n$. Then $B:=\cap_nB_n\in\Sigma_\X$ contains $x$ and has measure zero, therefore implying $\{x\}\in\bar \Sigma_\X$.

\noindent{(ii)} By the properties of liftings, $\ell(E)$ is not empty if and only if $\mu_\X(E)>0$. Thus if both $\ell(E)\cap A_i=\ell(E\cap A_i')$ and $\ell(\X\setminus E)\cap A_i=\ell((\X\setminus E)\cap A_i')$ are not empty, then both $E\cap A_i'$ and $(\X\setminus E)\cap A_i'$ have positive measure, contradicting the definition of $A_i'$.

Now if $f\in L^\infty(\mu_\X)$ is a simple function, say $f=\sum_j\alpha_i\1_{E_j}$, then $\ell(f)=\sum_j\alpha_j\1_{\ell  E_j}$ and the fact that $\ell (f)$ is constant on $A_i$ follows by what just proved. The conclusion for general $f\in L^\infty(\mu_\X)$ then follows by a density argument noticing that if $f_n\to f$ in $L^\infty(\mu_\X)$, then $\ell(f_n)\to \ell(f)$ uniformly.
\end{proof}
The sets $A_i$ given by this lemma are uniquely determined by $\ell$ and pairwise disjoint: we shall call them \emph{lifted atoms}.

\bigskip

Recall that in Section \ref{se:mss2} we fixed an $L^\infty(\mu_\X)$-normed module $\mathscr M$, its lifting  $\ell_\X\mathscr M$ and the corresponding lifting map $\ell_\X$. 

Let us describe $\ell_\X\mathscr M$ as space of sections; once this will be done an analogous description for $\mathscr M$ will be easy to obtain via passage to the quotient. For any $x\in\X$ we put
\[
\ell_\X\mathscr M_x:=\left\{
\begin{array}{ll}
\1_{\{x\}}\cdot\ell_\X\mathscr M,\qquad&\text{if }\{x\}\in\bar\Sigma_\X,\\
\1_{A_i}\cdot\ell_\X\mathscr M,\qquad&\text{if $x\in A_i$ for some $i\in I$.}
\end{array}
\right.
\]
The continuity of the product operation and the fact that characteristic functions are idempotent imply that $\ell_\X\mathscr M_x$ is a submodule of $\ell_\X\mathscr M$ and in particular a Banach space. The norm $\|\cdot\|_x$ on $\ell_\X\mathscr M_x$ is given by
\begin{equation}
\label{eq:normafibra}
\|v\|_x:= |v|(x)\qquad\forall v\in\ell_\X\mathscr M.
\end{equation}
Also,  given any \(v\in\ell_\X{\mathscr M}\), we shall define the `value of $v$ at $x$' as
\[
v_x\coloneqq\left\{
\begin{array}{ll}
\1_{\{x\}}\cdot v,\qquad&\text{if }\{x\}\in\bar\Sigma_\X,\\
\1_{A_i}\cdot v,\qquad&\text{if $x\in A_i$ for some $i\in I$,}
\end{array}\right.\quad \in\ell_\X{\mathscr M}_x,\qquad\text{ for every }x\in\X.
\]
Notice that $\|v_x\|_x=\|v\|_x=|v|(x)$. At this stage, to any $v\in\ell_\X\mathscr M$ is naturally associated  the  map \(\X\ni x\mapsto v_x\in\ell_\X{\mathscr M}_x\) and the construction ensures that $(\alpha_1v_1+\alpha_2v_2)_x=\alpha_1v_{1,x}+\alpha_2v_{2,x}$ for every $x\in\X$, $v_1,v_2\in\ell_\X{\mathscr M}$ and $\alpha_1,\alpha_2\in\R$.

It follows that the collection of Banach spaces $\ell_\X{\mathscr M}_\star$ and the set of test sections $\{x\mapsto v_x:v\in\ell_\X{\mathscr M}\}$ is a strong Banach bundle over $(\X,\bar\Sigma_\X,\{\varnothing\})$ in the sense of Definition \ref{def:strbb}. It is then a simple exercise to check that 
\begin{equation}
\label{eq:fibrelift}
\text{$\ell_\X{\mathscr M}\ni v\mapsto v_\cdot\in L^\infty_{\sf str}(\ell_\X{\mathscr M}_\star;\bar\Sigma_\X,\{\varnothing\})$ is an isomorphism of $\mathcal L^\infty(\bar\Sigma_\X)$-normed modules}.
\end{equation}
Indeed, the only non-trivial thing to check is surjectivity: this follows from the requirement $(ii)$ in the definition of strongly measurable section in conjunction with the completeness of $\ell_\X\mathscr M$ and the glueing property.

\begin{example}[Fibres of $\ell L^\infty(\mu)$]\label{ex:fibreR}{\rm
Let $(\X,\Sigma,\mu)$ be a $\sigma$-finite measure space and $\ell:\Sigma\to\bar\Sigma$ a lifting of $\mu$. Since $L^\infty(\mu)$ is, trivially, an $L^\infty(\mu)$-normed module, we can consider its lifting $(\ell L^\infty(\mu),\ell)$ in the sense of Theorem \ref{thm:lift_mod} and it is clear that an explicit realization of this lifting is given by the image of $L^\infty(\mu)$ in $\mathcal L^\infty(\bar\Sigma)$ via the lifting $\ell$ as defined after Theorem \ref{thm:ext_pre-lift}, with such $\ell$ as lifting map. Notice that in this realization we typically have $\ell L^\infty(\mu)\subsetneq \mathcal L^\infty(\bar\Sigma)$ because for any $A,B\in\bar\Sigma$ with $\mu(A\Delta B)=0$ at most one between $\1_A,\1_B$ is in $\ell L^\infty(\mu)$.

With this realization in mind and for fixed $x\in\X$ we claim that the map
\begin{equation}
\label{eq:fibreR}
\begin{array}{ccc}
\ell L^\infty(\mu)_x\quad&\to&\quad\R \\
f\quad&\mapsto&\quad f(x)
\end{array}
\end{equation}
is an isomorphism. Linearity is clear and norm preservation follows from the definition \eqref{eq:normafibra}. For surjectivity we notice that the properties of lifting ensure that if ${\bf 1}\in L^\infty(\mu)$ is the function $\mu$-a.e.\ equal to 1, then $\ell{\bf 1}\equiv 1$ and thus in particular $(\ell{\bf 1}_x)(x)=1$. Injectivity is clear if $\{x\}\in\bar\Sigma$, otherwise we use $(ii)$ of Lemma \ref{le:la} to conclude.
}\fr\end{example}
It is now easy to see the original module $\mathscr M$ as space of sections. Notice indeed that since what we just described  is a strong Banach bundle over  $(\X,\bar\Sigma_\X,\{\varnothing\})$, it is a fortiori a strong Banach bundle over  $(\X,\bar\Sigma_\X,\mathcal N_{\mu_\X})$, where $\mathcal N_{\mu_\X}\subset\bar\Sigma_\X$ is the $\sigma$-ideal of $\mu_\X$-negligible sets.

Hence the module   $L^\infty_{\sf str}(\ell_\X{\mathscr M}_\star;\bar\Sigma_\X,\mathcal N_{\mu_\X})\cong L^\infty_{\sf str}(\ell_\X{\mathscr M}_\star;\mu_\X)$ is well defined. We shall denote by $\pi_{\mu_\X}:L^\infty_{\sf str}(\ell_\X{\mathscr M}_\star;\bar\Sigma_\X,\{\varnothing\})\to  L^\infty_{\sf str}(\ell_\X{\mathscr M}_\star;\mu_\X)$ the projection map corresponding to the equivalence relation $\sim_{\mu_\X}$ defined as: $v_\cdot\sim_{\mu_\X} w_\cdot$ if and only if $\mu_\X(\{x:\|v_x-w_x\|_x>0\})=0$.

We then have the following result, that can be regarded as a `sectionwise version' of Lemma \ref{lem:quotient_lift}.
\begin{proposition} With the above notation and assumptions, we have that
\[
\text{the map $\pi_{\mu_\X}\circ\ell_\X:\mathscr M\to L^\infty_{\sf str}(\ell_\X{\mathscr M}_\star;\mu_\X)$ is an isomorphism of $L^\infty(\mu_\X)$-normed modules.}
\]
\end{proposition}
\begin{proof}
The construction ensures that $\pi_{\mu_\X}\circ\ell_\X$ is $L^\infty(\mu_\X)$-linear and preserves the pointwise norm, thus we only have to prove surjectivity. An arbitrary element of $L^\infty_{\sf str}(\ell_\X{\mathscr M}_\star;\mu_\X)$ can be written as  $\pi_{\mu_\X}(v_\cdot)$ for some $v_\cdot\in L^\infty_{\sf str}(\ell_\X{\mathscr M}_\star;\bar\Sigma_\X,\{\varnothing\})$ and \eqref{eq:fibrelift}  ensures that $v_\cdot$ is in fact the map $x\mapsto v_x$ for some $v\in\ell_\X\mathscr M$. We thus put $\tilde v:=\pi_{\mu_\X}(v)\in\Pi_{\mu_\X}(\ell_\X\mathscr M)\cong\mathscr M$, having used Lemma \ref{lem:quotient_lift} in the stated isomorphism, and notice that unwrapping the definitions we have  that $\pi_{\mu_\X}(\ell_\X(\tilde v))=\pi_{\mu_\X}(v_\cdot)\in  L^\infty_{\sf str}(\ell_\X{\mathscr M}_\star;\mu_\X)$, as desired.
\end{proof}

\subsubsection{The pullback}
In this section we study the relation between the pullback operation and liftings and use this knowledge to describe the pullback of a module as space of sections.

Let us fix two $\sigma$-finite measure spaces \((\X,\Sigma_\X,\mu_\X)\), \((\Y,\Sigma_\Y,\mu_\Y)\), liftings  \(\ell_\X\) and \(\ell_\Y\)  of
\(\mu_\X\) and \(\mu_\Y\) respectively and a map  \(\varphi\colon\Y\to\X\) that is $(\bar\Sigma_\Y,\bar\Sigma_\X)$-measurable and with \(\varphi_*\mu_\Y=\mu_\X\).

\begin{definition}[Compatible lifting through composition]\label{def:complif}
We say that \(\ell_\X\) and \(\ell_\Y\) are {compatible through \(\varphi\)} provided it holds that
\[
\ell_\Y\big(\varphi^{-1}(E)\big)=\varphi^{-1}\big(\ell_\X(E)\big),\quad\text{ for every }E\in\Sigma_\X.
\]
\end{definition}
It can be readily checked that $\ell_\X$ and $\ell_\Y$ are compatible  through \(\varphi\) if and only if
\begin{equation}\label{eq:compat_fcs}
\ell_\Y(f\circ\varphi)=\ell_\X(f)\circ\varphi,\quad\text{ for every }f\in L^\infty(\mu_\X).
\end{equation}
This can be proved by first inspecting simple functions and then by approximation.

The following is easily established:
\begin{proposition}[Existence of compatible liftings]\label{thm:exist_compat_lift}
Let \((\X,\Sigma_\X,\mu_\X)\) and \((\Y,\Sigma_\Y,\mu_\Y)\) be $\sigma$-finite measure spaces, \(\ell_\X\) a lifting of \(\mu_\X\) and \(\varphi\colon\Y\to\X\) a $(\bar\Sigma_\Y,\bar\Sigma_\X)$-measurable map such that \(\varphi_*\mu_\X=\mu_\Y\). 

Then there exists a lifting
\(\ell_\Y\) of \(\mu_\Y\) that is compatible with \(\ell_\X\) through \(\varphi\).
\end{proposition}
\begin{proof}
Define \(\Xi\coloneqq\{\varphi^{-1}(E)\,:\,E\in\bar\Sigma_\X\}\subset\bar\Sigma_\Y\). It can be readily checked that \(\Xi\) is a \(\sigma\)-subalgebra of \(\bar\Sigma_\Y\).
Define \(\phi\colon\Xi\to\bar\Sigma_\Y\) as \(\phi\big(\varphi^{-1}(E)\big)\coloneqq\varphi^{-1}\big(\ell_\X(E)\big)\) for every \(E\in\bar\Sigma_\X\).
Let us first prove that \(\phi\) is well-defined: given any \(E,F\in\bar\Sigma_\X\) satisfying \(\varphi^{-1}(E)=\varphi^{-1}(F)\), it holds that
\(\mu_\X(E\Delta F)=\mu_\Y\big(\varphi^{-1}(E)\Delta\varphi^{-1}(F)\big)=0\) and thus \(\ell_\X(E)\Delta\ell_\X(F)=\ell_\X(E\Delta F)=\varnothing\),
which implies that \(\ell_\X(E)=\ell_\X(F)\) and accordingly \(\varphi^{-1}\big(\ell_\X(E)\big)=\varphi^{-1}\big(\ell_\X(F)\big)\).
It is straightforward to check that \(\phi\) is a Boolean homomorphism. Let us now check that \(\phi\) is a pre-lifting. Given any \(E,F\in\bar\Sigma_\X\)
with \(\mu_\Y\big(\varphi^{-1}(E)\Delta\varphi^{-1}(F)\big)=0\), we have that \(\mu_\X(E\Delta F)=0\), so accordingly
\[
\phi\big(\varphi^{-1}(E)\big)=\varphi^{-1}\big(\ell_\X(E)\big)=\varphi^{-1}\big(\ell_\X(F)\big)=\phi\big(\varphi^{-1}(F)\big).
\]
Moreover, for any \(E\in\bar\Sigma_\X\) we have that
\[
\mu_\Y\big(\varphi^{-1}(E)\Delta\phi\big(\varphi^{-1}(E)\big)\big)=\mu_\Y\big(\varphi^{-1}(E)\Delta\varphi^{-1}\big(\ell_\X(E)\big)\big)
=\mu_\X\big(E\Delta\ell_\X(E)\big)=0.
\]
Therefore, \(\phi\) is a pre-lifting of \(\mu_\Y\). Extending it to a lifting via  Theorem \ref{thm:ext_pre-lift} we  conclude.
\end{proof}
In a moment we are going to study the pullback of the lifted module $\ell_\X\mathscr M$: this concept is defined via the following result, whose proof closely follows that of the analogue Theorem/Definition \ref{thm:defpb} and will therefore be omitted.
\begin{thmdef}[Pullback of an \(L^\infty(\Sigma_\X,\mathcal N_\X)\)-normed module]\label{thm:defpb2}
Let \((\X,\Sigma_\X)\) and \((\Y,\Sigma_\Y)\) be measurable spaces. Let \(\mathcal N_\X\subset\Sigma_\X\) and \(\mathcal N_\Y\subset\Sigma_\Y\)
be \(\sigma\)-ideals. Let \(\varphi\colon\Y\to\X\) be a measurable map such that \(\varphi^{-1}(N)\in\mathcal N_\Y\) for every \(N\in\mathcal N_\X\).
Let \(\mathscr M\) be an \(L^\infty(\Sigma_\X,\mathcal N_\X)\)-normed module. 

Then there exists a unique couple \((\varphi^*\mathscr M,\varphi^*)\), called
the \emph{pullback} of \(\mathscr M\) under \(\varphi\), where \(\varphi^*\mathscr M\) is an \(L^\infty(\Sigma_\Y,\mathcal N_\Y)\)-normed module and
\(\varphi^*\colon\mathscr M\to\varphi^*\mathscr M\) is a linear map such that the following conditions hold:
\begin{itemize}
\item[i)] The identity \(|\varphi^*v|=|v|\circ\varphi\) holds for every \(v\in\mathscr M\).
\item[ii)] $\{\varphi^*v\,:\,v\in\mathscr M\}$ generates  \(\varphi^*\mathscr M\).
\end{itemize}
Uniqueness is up to unique isomorphism: given any \((\mathscr N,T)\) with the same properties, there exists a unique 
 isomorphism \(\Phi\colon\varphi^*\mathscr M\to\mathscr N\) such that \(T=\Phi\circ\varphi^*\).
\end{thmdef}
We turn to the description of $\varphi^*(\ell_\X\mathscr M)$ as space of sections. Recall from the previous section that we have the collection $\ell_\X\mathscr M_\star$ of Banach spaces indexed by points in $\X$. We can then consider the collection $\ell_\X\mathscr M_{\varphi(\star)}$ of Banach spaces indexed by points in $\Y$, where the space corresponding to $y\in\Y$ is $\ell_\X\mathscr M_{\varphi(y)}$. We shall consider as test sections those of the kind $y\mapsto v_{\varphi(y)}$ for $v\in\ell_\X\mathscr M$ (recall \eqref{eq:fibrelift}). It is clear that these satisfy the requirements in Definition \ref{def:strbb} on $(\Y,\bar\Sigma_\Y,\{\varnothing\})$ and thus that what just described is a strong Banach bundle.

We then have the following result:
\begin{proposition}[Pullback of lifted modules]\label{prop:lift_and_pullback_mod} With the same assumptions and notation as in Section \ref{se:mss2} and above, we have:
\begin{equation}
\label{eq:equivpblift}
\Big(\,L^\infty_{\sf str}\big(\ell_\X\mathscr M_{\varphi(\star)};\bar \Sigma_\Y,\{\varnothing\}\big)\,,\,{\rm Pb}\,\Big)\cong\big(\varphi^*(\ell_\X\mathscr M),\varphi^*\big)
\end{equation}
where the pullback map ${\rm Pb}:\ell_\X\mathscr M\to L^\infty_{\sf str}(\ell_\X\mathscr M_{\varphi(\star)};\bar\Sigma_\Y,\{\varnothing\})$ is the one sending $v=(x\mapsto v_x)$ (recall \eqref{eq:fibrelift}) to $y\mapsto v_{\varphi(y)}$.
\end{proposition}
\begin{proof} By construction we have $|{\rm Pb}(v)(y)|=|v_{\varphi(y)}|=|v|(\varphi(y))$ for any $v\in \ell_\X\mathscr M$ and $y\in\Y$.  This and the definition of the strong Banach bundle $\ell_\X\mathscr M_{\varphi(\star)}$ on $(\Y,\bar\Sigma_\Y,\{\varnothing\})$ given above gives the conclusion.
\end{proof}
Much like in the previous section, we now observe that since the collection $\ell_\X\mathscr M_{\varphi(\star)}$ of Banach spaces together with the family of test sections $\{y\mapsto v_{\varphi(y)}:v\in\ell_\X\mathscr M\}$ is a strong Banach bundle on $(\Y,\bar\Sigma_\Y,\{\varnothing\})$, it is a fortiori a strong Banach bundle over $(\Y,\bar\Sigma_\Y,\mathcal N_{\mu_\Y})$, where $\mathcal N_{\mu_\Y}\subset\bar\Sigma_\Y$ is the collection of $\mu_\Y$-negligible sets. 

Hence $L^\infty_{\sf str}(\ell_\X{\mathscr M}_{\varphi(\star)};\bar\Sigma_\Y,\mathcal N_{\mu_\Y})\cong L^\infty_{\sf str}(\ell_\X{\mathscr M}_{\varphi(\star)};\mu_\Y)$ is well defined. We shall denote by $\pi_{\mu_\Y}:L^\infty_{\sf str}(\ell_\X{\mathscr M}_{\varphi(\star)};\bar\Sigma_\Y,\{\varnothing\})\to  L^\infty_{\sf str}(\ell_\X{\mathscr M}_{\varphi(\star)};\mu_\Y)$ the projection map corresponding to the equivalence relation $\sim_{\mu_\Y}$ defined as: $v_\cdot\sim_{\mu_\Y} w_\cdot$ if and only if $\mu_\Y(\{y:\|v_y-w_y\|_y>0\})=0$.

With this said, we have:
\begin{theorem}[Pullback of modules]With the same assumptions and notation as in Section \ref{se:mss2} and above, we have:
\begin{equation}
\label{eq:isopb}
\Big(\,L^\infty_{\sf str}\big(\ell_\X\mathscr M_{\varphi(\star)};\mu_\Y\big)\,,\,\pi_{\mu_\Y}\circ{\rm Pb}\circ\ell_\X\,\Big)\ \cong\ \big(\varphi^*\mathscr M,\varphi^*\big),
\end{equation}
where ${\rm Pb}:\ell_\X\mathscr M\to L^\infty_{\sf str}(\ell_\X\mathscr M_{\varphi(\star)};\bar\Sigma_\Y,\{\varnothing\})$ is  as in Proposition \ref{prop:lift_and_pullback_mod} above.
\end{theorem}
\begin{proof} For any $v\in\mathscr M $ we have
\[
\begin{split}
|\pi_{\mu_\Y}({\rm Pb}(\ell_\X(v)))|=|{\rm Pb}(\ell_\X(v))|=|\ell_\X(v)|\circ\varphi=|v|\circ\varphi\qquad\mu_\Y-a.e..
\end{split}
\]
Also, elements of the form $\ell_\X(v)$ with $v\in\mathscr M$ generate $\ell_\X\mathscr M$ (by Theorem \ref{thm:lift_mod}) and thus elements of the form ${\rm Pb}(\ell_\X(v))$ with $v\in\mathscr M$ generate $L^\infty_{\sf str}(\ell_\X\mathscr M_{\varphi(\star)};\bar \Sigma_\Y,\{\varnothing\})$ (by Proposition \ref{prop:lift_and_pullback_mod} above). It follows that elements of the form  $\pi_{\mu_\Y}({\rm Pb}(\ell_\X(v)))$ with $v\in\mathscr M$ generate $L^\infty_{\sf str}(\ell_\X\mathscr M_{\varphi(\star)};\mu_\Y)$, so the conclusion follows from the uniqueness part of Theorem \ref{thm:defpb2}.
\end{proof}
We notice also that under the same assumptions of this last theorem the map $\varphi^*(\ell_\X(v))\mapsto\ell_\Y(\varphi^* v)$ uniquely extends to an $\mathcal L^\infty(\bar\Sigma_\Y)$-linear continuous map from $\varphi^*(\ell_\X\mathscr M)$ to $\ell_\Y(\varphi^*\mathscr M)$ and such extension is an isomorphism. In other words, the diagram
\begin{equation}\label{eq:pbcomm}\begin{tikzcd}
\ell_\Y(\varphi^*\mathscr M)\sim\varphi^*(\ell_\X\mathscr M) & \ell_\X\mathscr M\arrow[l,swap,"\varphi^*"]\\
\varphi^*\mathscr M\arrow[u,"\ell_\Y"] &\mathscr M \arrow[u,swap,"\ell_\X"]  \arrow[l,"\varphi^*"] 
\end{tikzcd}\end{equation}
commutes.

\subsubsection{The dual}

Here we study how to interpret the dual of a module as space of sections of a weak Banach bundle. We shall again rely on the concept of lifting of module, but, unlike the discussion for the pullback, we will not describe the dual of the lifting (whose description in terms of sections seems hard to achieve) and will instead directly proceed with the description of the dual of $\mathscr M$.

\bigskip

We start with some definitions. Let $(\X,\Sigma,\mathcal N)$ be an enhanced measurable space and $\mathscr M,\mathscr N$ two $L^\infty(\Sigma,\mathcal N)$-normed modules.

A \emph{morphism} from $\mathscr M$ to $\mathscr N$ is a map that is continuous and  $L^\infty(\Sigma,\mathcal N)$-linear. In general, to such a morphism $T$ we cannot associate a notion of `pointwise norm'. This is possible, though, in two key circumstances:
\begin{itemize}
\item[i)] If $\mathcal N=\mathcal N_\mu$ is the collection of $\mu$-negligible sets for some given $\sigma$-finite measure $\mu$. In this case we define $|T|\in L^\infty(\mu)$ as
\begin{equation}
\label{eq:normessup}
|T|:=\underset{\genfrac{}{}{0pt}{2}{v\in \mathscr M:}{\|v\|_{\mathscr M}\leq 1}}{\rm ess\,sup\,} |T(v)|,
\end{equation}
where the essential supremum is intended w.r.t.\ $\mu$.
\item[ii)] If $\mathcal N=\{\varnothing\}$. In this case we define $|T|:\X\to[0,\infty)$ as
\begin{equation}
\label{eq:normpunt}
|T|(x):=\sup_{\genfrac{}{}{0pt}{2}{v\in \mathscr M:}{\|v\|_{\mathscr M}\leq 1}}|T(v)|(x).
\end{equation}
\end{itemize}
Notice that in the latter case $|T|$ is not necessarily measurable. The next proposition shows anyway that this is the case when we `lift' a morphism:

\begin{proposition}[Lifting of homomorphisms]\label{thm:lift_hom}
Let \((\X,\Sigma,\mu)\) be a $\sigma$-finite measure space and \(\ell:\Sigma\to\bar\Sigma\) a lifting of \(\mu\), $\bar\Sigma$ being the completion of $\Sigma$. Let \(T\colon\mathscr M\to\mathscr N\) be
a morphism between two \(L^\infty(\mu)\)-normed modules \(\mathscr M\) and \(\mathscr N\). Then there is a unique morphism
\(\ell T\colon\ell\mathscr M\to\ell\mathscr N\) of  \(\mathcal L^\infty(\bar \Sigma)\)-normed modules such that
\begin{equation}\label{eq:lift_hom}\begin{tikzcd}
\ell\mathscr M \arrow[r,"\ell T"] & \ell\mathscr N\\
\mathscr M \arrow[r,"T"] \arrow[u,"\ell"] &\mathscr N \arrow[u,swap,"\ell"] 
\end{tikzcd}\end{equation}
is a commutative diagram. Moreover, it holds that
\begin{equation}\label{eq:ptwse_norm_lift_Phi}
|\ell T|=\ell(|T|)\in\mathcal L^\infty(\bar\Sigma).
\end{equation}
\end{proposition}
\begin{proof} By $\mathcal L^\infty(\bar\Sigma)$-linearity and the well posedness of the glueing it is clear that we are forced to define 
\[
(\ell T)(\bar v)\coloneqq\sum_{n\in\N}\1_{E_n}\cdot\ell\big(T(v_n)\big)\in\ell\mathscr N,
\quad\text{ for every }\bar v=\sum_{n\in\N}\1_{E_n}\cdot\ell(v_n)\in\ell\mathscr M,
\]
with $(v_n)\subset\mathscr M$ uniformly bounded and $(E_n)\subset\Sigma$ partition of $\X$. From
\[\begin{split}
\big|(\ell T)(\bar v)\big|&=\sum_{n\in\N}\1_{E_n}\big|\ell\big( T(v_n)\big)\big|=
\sum_{n\in\N}\1_{E_n}\,\ell\big(| T(v_n)|\big)\leq\sum_{n\in\N}\1_{E_n}\,\ell\big(| T||v_n|\big)\\
&=\ell(| T|)\sum_{n\in\N}\1_{E_n}\,\ell(|v_n|)=\ell(| T|)\sum_{n\in\N}\1_{E_n}\big|\ell(v_n)\big|=\ell(| T|)|\bar v|,
\end{split}\]
it follows that   the resulting operator \(\ell T\) is well-defined, linear, and continuous. Since the class of $\bar v\in \ell\mathscr M$ as  above is dense in $\ell\mathscr M$, there is a unique continuous extension, still denoted by \(\ell T\), to an operator from \(\ell\mathscr M\) to \(\ell \mathscr N\) and by the above it satisfies
\begin{equation}\label{eq:ineq_lift_Phi}
\big|(\ell T)(\bar v)\big|\leq\ell(| T|)|\bar v|,\quad\text{ for every }\bar v\in\ell\mathscr M.
\end{equation}
It is clear that $\ell T$ commutes with products of characteristics, thus by linearity and continuity it is immediate to verify that it is $\mathcal L^\infty(\bar\Sigma)$-linear. Also, from  \eqref{eq:ineq_lift_Phi} we see that to conclude that \eqref{eq:ptwse_norm_lift_Phi} holds it is enough to show that $|\ell T|\geq \ell(|T|)$. 

To prove this fix $\eps>0$ and use a glueing argument to find $v\in\mathscr M$ with $\|v\|_{\mathscr M}\leq 1$ so that \(\big| T(v)\big|\geq| T|-\eps\) holds \(\mu\)-a.e.\ on \(\X\). Thus we have 
\[\big|(\ell T)\big(\ell(v)\big)\big|=\big|\ell\big( T(v)\big)\big|=\ell\big(| T(v)|\big)\geq\ell(| T|)-\eps
\]
and since this holds at every point of $\X$ and $\|\ell(v)\|_{\ell\mathscr M}=\sup|\ell(v)|=\sup\ell(|v|)\leq 1$, by the arbitrariness of $\eps>0$ we conclude.
\end{proof}
In this part of the paper we shall work with the following notion of duality:
\begin{definition}[Dual of an \(L^\infty(\mu)\)-normed module]\label{def:dual_mod}
Let \((\X,\Sigma,\mu)\) be a $\sigma$-finite measure space and \(\mathscr M\) an \(L^\infty(\mu)\)-normed module. Then the \emph{module dual} \(\mathscr M^*\)
of \(\mathscr M\) is defined as the family of all the morphisms of \(L^\infty(\mu)\)-normed module from \(\mathscr M\) to \(L^\infty(\mu)\).  
\end{definition}
One can readily check that the dual space \(\mathscr M^*\) is an \(L^\infty(\mu)\)-normed module if endowed with the natural pointwise operations
and with the pointwise norm \(|\cdot|\) defined as in \eqref{eq:normessup}. We point out that this notion of module dual of
an \(L^\infty(\mu)\)-normed module, which is convenient for our purposes in this chapter, is different from the original one introduced in \cite[Definition 1.2.6]{Gigli14} and used in Section \ref{se:sep}. In particular, it has little to do with the Banach dual of $\mathscr M$.

The definition above has been chosen to ensure that the dual of an $L^\infty$-normed module is still $L^\infty$-normed (and not $L^1$-normed as in \cite{Gigli14}): remaining in the realm of $L^\infty$-normed modules is important if one wants to apply the lifting theory.

\begin{remark}[Relation with \(L^p\)-normed \(L^\infty\)-modules]\label{rmk:consist_L_p}{\rm Even though the above concept of dual differs from the one in  \cite{Gigli14}, it is still strictly related to it.

For $p\in[1,\infty)$ and $\mathscr M$ an $L^\infty(\mu)$-normed module, let  \({\sf C}_p(\mathscr M)\) be the completion of $\mathscr M$ w.r.t.\ the norm $v\mapsto \||v|\|_{L^p(\mu)}$ (more precisely: one should  restrict to the space of $v$'s with finite norm before completing). Then  \({\sf C}_p(\mathscr M)\)  has a natural structure of \(L^p(\mu)\)-normed \(L^\infty(\mu)\)-module. 

Conversely, given an \(L^p(\mu)\)-normed \(L^\infty(\mu)\)-module \(\mathscr M_p\), it holds that the `restricted' space made of those $v$'s in the $L^0$-completion of $\mathscr M_p$ (see \cite{Gigli14} for this concept)  with  $|v|\in L^\infty(\mu)$  is trivially an \(L^\infty(\mu)\)-normed module.  

It is then not hard to see that for any $L^\infty(\mu)$-normed module $\mathscr M$ we have
\begin{equation}\label{eq:consist_dual}
\mathscr M^*\cong{\sf R}\big({\sf C}_p(\mathscr M)^*\big),
\end{equation}
where the dual on the left is as in Definition \ref{def:dual_mod} and that on the right as in Definition \ref{def:dualm}.

We also point out that   \({\sf C}_p\) and \({\sf R}\) can be made into
functors which are equivalences of categories, one the inverse of the other. These claims can be proven by adapting the arguments in \cite[Appendix B]{LP18}.
\fr}\end{remark}
We are now ready to describe the dual of $\mathscr M$ as space of sections. Recall that for a given lifting $\ell_\X$ of $\mu_\X$, the collection $\ell_\X\mathscr M_\star$ of fibres of $\ell_\X\mathscr M$ has been introduced in Section \ref{se:basem2}, thus we can also consider the collection $\ell_\X\mathscr M'_\star$ of their duals. Recalling  \eqref{eq:fibrelift}, we also see that the collection $\{x\mapsto v_x\in \ell_\X\mathscr M_x:v\in\ell_\X\mathscr M\}$ is a vector space. In particular, it satisfies the requirement $(a)$ in Definition \ref{def:wbb} and therefore  $\ell_\X\mathscr M'_\star$ with these test vectors is a weak Banach bundle on $(\X,\bar\Sigma_\X,\mathcal N_{\mu_\X})$, where $\mathcal N_{\mu_\X}\subset\bar\Sigma_\X$ is the collection of $\mu_\X$-negligible sets.

According to the discussion in Section \ref{se:wbb} we can therefore consider the space $\mathscr S_{\sf weak}(\ell_\X\mathscr M'_\star)$ of weakly measurable sections, its quotient  space $\mathscr S_{\sf weak}(\ell_\X\mathscr M'_\star)/\sim_{\sf weak}$ and the associated projection map, that we shall denote by $\pi_{\mu_\X}$. Then $L^\infty_{\sf weak}(\ell_\X{\mathscr M}'_\star;\mu_\X)\subset \mathscr S_{\sf weak}(\ell_\X\mathscr M'_\star)/\sim_{\sf weak}$ is the subspace of essentially bounded sections.

With this said we have:
\begin{theorem}\label{thm:dual2}
Let \((\X,\Sigma_\X,\mu_\X)\) be a $\sigma$-finite measure space, \(\ell_\X:\Sigma_\X\to\bar\Sigma_\X\) a lifting of \(\mu_\X\) and $\mathscr M$ an $L^\infty(\mu_\X)$-normed module.

Then $L^\infty_{\sf weak}(\ell_\X{\mathscr M}'_\star;\mu_\X)$ is an $L^\infty(\mu_\X)$-normed module and the map
\[
\begin{array}{rccc}
{\sf I}:&L^\infty_{\sf weak}(\ell_\X{\mathscr M}'_\star;\mu_\X)\qquad&\to&\qquad\mathscr M^*\\
&\pi_{\mu_\X}(\omega_\cdot) \qquad&\mapsto&\quad\big(v\quad\mapsto\quad [\langle \omega_\cdot,\ell_\X(v)_\cdot \rangle]\big)
\end{array}
\]
is an isomorphism of $L^\infty(\mu_\X)$-normed modules (in the above, in writing $\ell_\X(v)_\cdot$ we are thinking of $\ell_\X(v)\in\ell_\X\mathscr M$ as an element of $L^\infty_{\sf str}(\ell_\X{\mathscr M}_\star;\bar\Sigma_\X,\{\varnothing\})$ as in \eqref{eq:fibrelift} -- also, by $[f]$ we denote the equivalence class up to $\mu_\X$-a.e.\ equality of the measurable function $f$).
\end{theorem}
\begin{proof} The very definition of $L^\infty_{\sf weak}(\ell_\X{\mathscr M}'_\star;\mu_\X)$ and of pointwise norm on it ensure that ${\sf I}$ is well defined, $L^\infty(\mu_\X)$-linear and that $|{\sf I}(\omega)|\leq |\omega|$ $\mu_\X$-a.e.\ for every $\omega\in L^\infty_{\sf weak}(\ell_\X{\mathscr M}'_\star;\mu_\X)$.

To conclude it is therefore sufficient to show that it is surjective and that  $|{\sf I}(\omega)|\geq |\omega|$ holds  $\mu_\X$-a.e.\ for every $\omega\in L^\infty_{\sf weak}(\ell_\X{\mathscr M}'_\star;\mu_\X)$ (notice that this shows that $L^\infty_{\sf weak}(\ell_\X{\mathscr M}'_\star;\mu_\X)$ is isometric to $\mathscr M^*$ and in particular complete, which is the only non-trivial claim in saying that $L^\infty_{\sf weak}(\ell_\X{\mathscr M}'_\star;\mu_\X)$ is an $L^\infty(\mu_\X)$-normed module).

Thus let \(L\in\mathscr M^*\) and $\ell_\X L:\ell_\X\mathscr M\to\ell_\X  L^\infty(\mu_\X)$ be its lifting as in Proposition \ref{thm:lift_hom} above. Since $\ell_\X L$ commutes with the product by $\1_{\{x\}}$ and $\1_{A_i}$ it induces a linear and continuous map, that we shall denote by $\ell_\X L_x$, from $\ell_\X\mathscr M_x$ to $\ell_\X  L^\infty(\mu_\X)_x\cong \R$, having used the isomorphism in Example \ref{ex:fibreR}. In other words,  $\bar\omega(x) \coloneqq\ell_\X L_x$ is an element of $(\ell_\X\mathscr M_x)'$ and 
for any  $v\in \mathscr M$ and \(x\in\X\) we have
\begin{equation}
\label{eq:quasifatto}
\langle\bar\omega(x),\,\ell_\X(v)_x\rangle=\ell_\X L_x(\ell_\X(v)_x)\stackrel{\eqref{eq:fibreR}}=(\ell_\X L)(\ell_\X(v) )(x)\overset{\eqref{eq:lift_hom}}=\ell_\X\big(L(v)\big)(x).
\end{equation}
This proves that  $x\mapsto \langle\bar\omega(x),\,\ell_\X(v)_x\rangle$ is measurable and thus, by the arbitrariness of $v\in\mathscr M$, that $\bar\omega\in \mathscr S_{\sf weak}(\ell_\X{\mathscr M}'_\star)$. Hence we can define
\(\omega\coloneqq\pi_{\mu_\X}(\bar\omega)\in  \mathscr S_{\sf weak}(\ell_\X{\mathscr M}'_\star)/\sim_{\sf weak}\). The inequality
\begin{equation}
\label{eq:fin}
|\omega|\stackrel{\eqref{eq:defnw}}=\underset{\genfrac{}{}{0pt}{2}{v\in \mathscr M:}{\|v\|_{\mathscr M}\leq 1}}{\rm ess\,sup\,}\langle\bar \omega, \ell_\X(v)_\cdot\rangle\stackrel{\eqref{eq:quasifatto}}=\underset{\genfrac{}{}{0pt}{2}{v\in \mathscr M:}{\|v\|_{\mathscr M}\leq 1}}{\rm ess\,sup\,}\ell_\X\big(L(v)\big)=\underset{\genfrac{}{}{0pt}{2}{v\in \mathscr M:}{\|v\|_{\mathscr M}\leq 1}}{\rm ess\,sup\,}L(v)\leq|L|,
\end{equation}
valid $\mu_\X$-a.e.\ shows that $|\omega|\in L^\infty(\mu_\X)$, and thus that $\omega\in L^\infty_{\sf weak}(\ell_\X{\mathscr M}'_\star;\mu_\X)$. Also, unwrapping the definitions it is  immediate to check that \(\langle {\sf I}(\omega),v\rangle=L(v)\in L^\infty(\mu_\X)\) for every \(v\in\mathscr M\), i.e.\ that \({\sf I}(\omega)=L\). This proves that \({\sf I}\) is surjective and then \eqref{eq:fin} shows that $|\omega|\leq|{{\sf I}}(\omega)|$, as desired.
\end{proof}
We point out that the analogue of Theorem \ref{thm:dual2} for the space of \(L^p\)-sections of a \emph{measurable Banach \(\B\)-bundle} was recently
obtained in \cite[Theorem 3.10]{LPV22}.

\subsubsection{The dual of the pullback}

We now turn to the description of the dual of the pullback in terms of sections: as in the last paragraph, since we are describing a dual module, the concept of weak Banach bundle is the relevant one.

\medskip

We shall make use of the following property of pullbacks: it is the `$L^\infty$-analogue' of Proposition \ref{prop:univ_prop_pullback-pol} and its proof will be omitted.

\begin{proposition}\label{prop:univ_prop_pullback-2} With the notation and assumptions as in Section \ref{se:mss2}, the following holds.

Suppose that \(T\colon\mathscr M\to L^\infty(\mu_\Y)\) is a linear operator for which there exists a function \(g\in L^\infty(\mu_\Y)\) such that \(|T(v)|\leq g|v|\circ\varphi\)
holds for every \(v\in\mathscr M\). 

Then there exists a unique linear and continuous map \(\hat T\colon\varphi^*\mathscr M\to L^\infty(\mu_\Y)\) such that
\[
\hat T(\varphi^*v)=T(v),\quad\text{ for every }v\in\mathscr M.
\]
Moreover, it holds that \(|\hat T(V)|\leq g|V|\) for every \(V\in\varphi^*\mathscr M\).
\end{proposition}
Let $\mathscr M$ be an $L^\infty(\mu_\X)$-normed module and, as before, $\ell_\X\mathscr M_\star$ the collection of fibres of $\ell_\X\mathscr M$ on $\X$ and  $\ell_\X\mathscr M'_\star$ that of their duals. We shall denote by $\ell_\X\mathscr M_{\varphi(\star)}$, $\ell_\X\mathscr M'_{\varphi(\star)}$ the collections of Banach spaces on $\Y$ assigning to $y\in\Y$ the space $\ell_\X\mathscr M_{\varphi(y)}$,  $\ell_\X\mathscr M'_{\varphi(y)}$ respectively. Finally for $v\in\mathscr M$ we denote by $\ell_\X(v)_{\varphi(\cdot)}\in\mathscr S(\ell_\X\mathscr M_{\varphi(\star)})$ the map $y\mapsto \ell_\X(v)_{\varphi(y)}$ (that is well defined by \eqref{eq:fibrelift}). It is clear that the collection $\{\ell_\X(v)_{\varphi(\cdot)}:v\in\mathscr M\}\subset \mathscr S(\ell_\X\mathscr M_{\varphi(\star)})$ is a vector space. In particular, it satisfies the requirement $(a)$ in Definition \ref{def:wbb} and therefore  $\ell_\X\mathscr M'_{\varphi(\star)}$ with these test vectors is a weak Banach bundle on $(\Y,\bar\Sigma_\Y,\mathcal N_{\mu_\Y})$, where $\mathcal N_{\mu_\Y}\subset\bar\Sigma_\Y$ is the collection of $\mu_\Y$-negligible sets.

According to the discussion in Section \ref{se:wbb} we can thus consider the space $\mathscr S_{\sf weak}(\ell_\X\mathscr M'_{\varphi(\star)})$ of weakly measurable sections, its quotient space $\mathscr S_{\sf weak}(\ell_\X\mathscr M'_{\varphi(\star)})/\sim_{\sf weak}$  and the associated projection map, that we denote by $\pi_{\mu_\Y}$. Then $L^\infty_{\sf weak}(\ell_\X{\mathscr M}'_{\varphi(\star)};\mu_\Y)\subset \mathscr S_{\sf weak}(\ell_\X\mathscr M'_{\varphi(\star)})/\sim_{\sf weak}$ is the subspace of essentially bounded sections.

Now observe that for $\pi_{\mu_\Y}(\omega_\cdot)\in L^\infty_{\sf weak}(\ell_\X{\mathscr M}'_{\varphi(\star)};\mu_\Y)$ and $v\in\mathscr M$ we have
\[
\begin{split}
|\langle \omega_\cdot,\ell_\X(v)_{\varphi(\cdot)}\rangle|\leq |\omega_\cdot|\,|\ell_\X(v)_{\varphi(\cdot)}|= |\omega_\cdot|\,|\ell_\X(v)_{\cdot}|\circ \varphi=|\omega_\cdot|\,| v|\circ \varphi,\qquad\mu_\Y-a.e.,
\end{split}
\]
therefore recalling Proposition \ref{prop:univ_prop_pullback-2} we see that there is a unique element $\mathcal I(\omega_\cdot)$ of $(\varphi^*\mathscr M)^*$ such that 
\[
\langle \mathcal I(\omega_\cdot),\varphi^*v \rangle=\pi_{\mu_\Y}(\langle \omega_\cdot,\ell_\X(v)_{\varphi(\cdot)}\rangle)\qquad\forall v\in\mathscr M
\]
and this element satisfies $|\mathcal I(\omega_\cdot)|\leq |\omega_\cdot|$ $\mu_\Y$-a.e..

We then have:
\begin{theorem}[The dual of the pullback: fibrewise characterization]\label{thm:dpb2}
With the  notation and assumptions as in Section \ref{se:mss2} and above, the space  $L^\infty_{\sf weak}(\ell_\X{\mathscr M}'_{\varphi(\star)};\mu_\Y)$ is an $L^\infty(\mu_\Y)$-normed module and the map
\[
\begin{array}{rccc}
\mathcal I:&L^\infty_{\sf weak}(\ell_\X{\mathscr M}'_{\varphi(\star)};\mu_\Y)\qquad&\to&\qquad(\varphi^*\mathscr M)^*
\end{array}
\]
is an isomorphism of $L^\infty(\mu_\Y)$-normed modules.
\end{theorem} 
\begin{proof}
$L^\infty(\mu_\Y)$-linearity is obvious and the inequality $|\mathcal I(\omega_\cdot)|\leq |\omega_\cdot|$ valid $\mu_\Y$-a.e.\ that we already proved yields also continuity.

To conclude we thus need to prove surjectivity and the bound  $|\mathcal I(\omega_\cdot)|\geq |\omega_\cdot|$ (as in the proof of Theorem \ref{thm:dual2} this also gives the completeness of $L^\infty_{\sf weak}(\ell_\X{\mathscr M}'_{\varphi(\star)};\mu_\Y)$). Thus let $L\in(\varphi^*\mathscr M)^*$, consider its lifting $\ell_\Y L$ as in Proposition \ref{thm:lift_hom} and notice that keeping the isomorphism in \eqref{eq:isopb} into account we can see $\ell_\Y L$ as an $\mathcal L^\infty(\bar\Sigma_\Y)$-linear and continuous map from $L^\infty_{\sf str}(\ell_\X\mathscr M_{\varphi(\star)};\bar\Sigma_\Y,\{\varnothing\})$ to $\mathcal L^\infty(\bar\Sigma_\Y)$. Evaluate such map at $y\in\Y$ to deduce that $\ell_\Y L_y$ is a map from $L^\infty_{\sf str}(\ell_\X\mathscr M_{\varphi(\star)};\bar\Sigma_\Y,\{\varnothing\})_y\sim \ell_\X\mathscr M_{\varphi(y)}$ (the isomorphism being a rather trivial consequence of the definitions) to $\mathcal L^\infty(\bar\Sigma_\Y)_y\sim \R$. In other words $\bar\omega(y):=\ell_\Y L_y\in \ell_\X\mathscr M_{\varphi(y)}'$ and for any $v\in\mathscr M$ and $y\in\Y$ we have
\begin{equation}
\label{eq:quasifatto2}
\begin{split}
\langle\bar\omega(y),\ell_\X(v)_{\varphi(y)}\rangle
\stackrel{\phantom{\eqref{eq:fibreR}}}=&\langle \ell_\Y L_y,{\rm Pb}(\ell_\X(v))_{y}\rangle \stackrel{\eqref{eq:pbcomm}}= \langle \ell_\Y L_y,\ell_\Y(\varphi^*v)_y\rangle\\
\stackrel{\eqref{eq:fibreR}}=&\langle \ell_\Y L,\ell_\Y(\varphi^*v)\rangle(y)
\stackrel{\eqref{eq:lift_hom}}=\ell_\Y(\langle L,\varphi^*v\rangle)(y).
\end{split}
\end{equation}
This proves that  $y\mapsto \langle\bar\omega(y),\,\ell_\X(v)_{\varphi(y)}\rangle$ is measurable and thus, by the arbitrariness of $v\in\mathscr M$, that $\bar\omega\in \mathscr S_{\sf weak}(\ell_\X{\mathscr M}'_{\varphi(\star)})$. Hence we can define
\(\omega\coloneqq\pi_{\mu_\Y}(\bar\omega)\in \mathscr S_{\sf weak}(\ell_\X{\mathscr M}'_{\varphi(\star)})/\sim_{\sf weak}\). The inequality
\begin{equation}
\label{eq:fin2}
|\omega|\stackrel{\eqref{eq:defnw}}=\underset{\genfrac{}{}{0pt}{2}{v\in \mathscr M:}{\|v\|_{\mathscr M}\leq 1}}{\rm ess\,sup\,}\langle\bar \omega, \ell_\X (v)_{\varphi(\cdot)}\rangle\stackrel{\eqref{eq:quasifatto2}}=\underset{\genfrac{}{}{0pt}{2}{v\in \mathscr M:}{\|v\|_{\mathscr M}\leq 1}}{\rm ess\,sup\,}\ell_\Y\big(L(\varphi^*v)\big)=\underset{\genfrac{}{}{0pt}{2}{v\in \mathscr M:}{\|v\|_{\mathscr M}\leq 1}}{\rm ess\,sup\,}L(\varphi^*v)\leq|L|,
\end{equation}
valid $\mu_\Y$-a.e.\ shows that $|\omega|\in L^\infty(\mu_\Y)$, and thus that $\omega\in L^\infty_{\sf weak}(\ell_\X{\mathscr M}'_{\varphi(\star)};\mu_\Y)$. Also, unwrapping the definitions it is  immediate to check that \(\langle{\mathcal I}(\omega),\varphi^*v\rangle=L(\varphi^* v)\in L^\infty(\mu_\Y)\) for every \(v\in\mathscr M\) and thus  that \({\mathcal I}(\omega)=L\). This proves that \(\mathcal I\) is surjective and then \eqref{eq:fin2} shows that $|\omega|\leq|{\mathcal I}(\omega)|$, as desired.
\end{proof}

\appendix

\section{\texorpdfstring{$(\varphi^*\mathscr M)^*$}{(phi star M)star}  as space of local maps}
We give here an alternative description of $(\varphi^*\mathscr M)^*$. Here we intend duality in the sense of Definition \ref{def:dualm}.

Let \((\X,\Sigma_\X,\mu_\X)\) and \((\Y,\Sigma_\Y,\mu_\Y)\) be   $\sigma$-finite measure spaces,  \(\varphi\colon\Y\to\X\) a measurable map with
\(\varphi_*\mu_\Y\ll\mu_\X\) and \(\mathscr M\) an $L^p(\mu_\X)$-normed module. 

We denote by
\(\textsc{Hom}_{\rm loc}\big(\mathscr M;L^1(\mu_\Y)\big)\) the space of all linear operators \(T\colon\mathscr M\to L^1(\mu_\Y)\)
for which there exists a function $g\in L^q(\mu_\Y)^+$ satisfying 
\[
|T(v)|\leq g|v|\circ\varphi\qquad\text{ in the \(\mu_\Y\)-a.e.\ sense for every \(v\in\mathscr M\).}
\]
To any such \(T\), we associate the function $|T|\in L^q(\mu_\Y)$ defined as the essential infimum of the $g$'s as above, or equivalently as
\begin{equation}\label{eq:ptwse_norm_Hom}
|T|\coloneqq\underset{\genfrac{}{}{0pt}{2}{v\in {\mathscr M}:}{|v|\leq1\ \mu_\X-a.e.}}{\rm ess\,sup\,}T(v)\in L^q(\mu_\Y).
\end{equation}
One can easily check that \(\textsc{Hom}_{\rm loc}\big(\mathscr M;L^1(\mu_\Y)\big)\) is an $L^q(\mu_\Y)$-normed module if endowed with such pointwise norm and
\[\begin{split}
(T+S)(v)\coloneqq T(v)+S(v),&\quad\text{ for every }T,S\in\textsc{Hom}_{\rm loc}\big(\mathscr M;L^1(\mu_\Y)\big)\text{ and }v\in\mathscr M,\\
(f\cdot T)(v)\coloneqq f\,T(v),&\quad\text{ for every }f\in L^\infty(\mu_\Y)\text{ and }T\in\textsc{Hom}_{\rm loc}\big(\mathscr M;L^1(\mu_\Y)\big).
\end{split}\]
\begin{theorem}[Dual of the pullback as space of local maps]\label{thm:IIIpol} With the above assumptions and notation, the following holds.

The map \({\rm I}\colon(\varphi^*\mathscr M)^*\to\textsc{Hom}_{\rm loc}\big(\mathscr M;L^1(\mu_\Y)\big)\)
 given by
\[
{\rm I}(L)(v)\coloneqq L(\varphi^*v)\in L^1(\mu_\Y),\quad\text{ for every }L\in(\varphi^*\mathscr M)^*\text{ and }v\in\mathscr M
\]
is an isomorphism of $L^q(\mu_\Y)$-normed modules.  
\end{theorem}
\begin{proof}
Given any \(L\in(\varphi^*\mathscr M)^*\) and \(v\in\mathscr M\), we can estimate \(\big|{\rm I}(L)(v)\big|\leq|L||\varphi^*v|=|L||v|\circ\varphi\)
in the \(\mu_\Y\)-a.e.\ sense. Since \(|L|\in L^q(\mu_\Y)\), we deduce that \({\rm I}(L)\in\textsc{Hom}_{\rm loc}\big(\mathscr M;L^1(\mu_\Y)\big)\)
and that \(|{\rm I}(L)|\leq|L|\) holds \(\mu_\Y\)-a.e.. The resulting operator \({\rm I}\colon(\varphi^*\mathscr M)^*\to
\textsc{Hom}_{\rm loc}\big(\mathscr M;L^1(\mu_\Y)\big)\) is linear by construction. Moreover, given any element
\(T\in\textsc{Hom}_{\rm loc}\big(\mathscr M;L^1(\mu_\Y)\big)\), we infer from Proposition \ref{prop:univ_prop_pullback-pol}
that there exists \(L\in(\varphi^*\mathscr M)^*\) such that \(L(\varphi^*v)=T(v)\) for every \(v\in\mathscr M\) and
\(|L|\leq|T|\) in the \(\mu_\Y\)-a.e.\ sense. This implies that \(T={\rm I}(L)\) and that \(|L|\leq|{\rm I}(L)|\) holds in the \(\mu_\Y\)-a.e.\ sense.
Consequently, the statement is achieved.
\end{proof}

\section{Sequential weak\texorpdfstring{$^*$}{star}-density of \texorpdfstring{$\varphi^*\mathscr M^*$}{phi star M star} in \texorpdfstring{$(\varphi^*\mathscr M)^*$}{(phi star M)star}}

We recalled in the introduction that in general $\varphi^*\mathscr M^*$ might be smaller than  $(\varphi^*\mathscr M)^*$. Here we show that, under general assumptions on the underlying measure structures, it is always sequentially weakly$^*$ dense.

\bigskip

Let \((\X,\Sigma_\X,\mu_\X)\) and \((\Y,\Sigma_\Y,\mu_\Y)\) be separable and  $\sigma$-finite measure spaces,  \(\varphi\colon\Y\to\X\) a measurable  map with
\(\varphi_*\mu_\Y\ll\mu_\X\) and \(\mathscr M\) an $L^p(\mu_\X)$-normed module. In \cite{Gigli14} it has been proved that there is a unique  \(L^\infty(\mu_\Y)\)-linear, pointwise norm preserving map  \({\sf I}_\varphi\colon\varphi^*\mathscr M^*\to(\varphi^*\mathscr M)^*\) satisfying
\[
\big\langle{\sf I}_\varphi(\varphi^*\omega),\varphi^*v\big\rangle=\langle\omega,v\rangle\circ\varphi,
\quad\text{ for every }\omega\in\mathscr M^*\text{ and }v\in\mathscr M.
\]
In the proof of the main result of this section it will be useful to have at disposal a left inverse of the pullback of functions via post-composition with $\varphi$. We thus define the map ${\rm Pr}:  L^\infty(\mu_\Y)\to L^\infty(\mu_\X)$ as
\[
{\rm Pr}(f)(x):=\frac{\d(\varphi_*(f\mu_\Y))}{\d\mu_\X}(x)\qquad\mu_\X-a.e.\ x.
\]
It is clear that for any $f\in L^\infty(\mu_\Y)$ and $g\in L^\infty(\mu_\X)$ it holds
\begin{equation}
\label{eq:prpr}
\begin{split}
|{\rm Pr}(f)|&\leq{\rm Pr}(|f|),\\
{\rm Pr}(f\,g\circ \varphi)&=g\,{\rm Pr}(f).
\end{split}
\end{equation}
The first of these shows that ${\rm Pr}$ is a contraction from $L^\infty(\mu_\Y)$ to $L^\infty(\mu_\X)$ and also from $L^1\cap L^\infty(\mu_\Y)$ to $L^1\cap L^\infty(\mu_\X)$ both endowed with the corresponding $L^1$ norm. It follows that for any $p\in[1,\infty]$ there is a unique linear and continuous extension, still denoted by ${\rm Pr},$ of ${\rm Pr}$ from $L^p(\mu_\Y)$ to $L^p(\mu_\X)$.

We then have:
\begin{theorem}[Sequential weak$^*$-density  of $\varphi^*\mathscr M^*$ in  $(\varphi^*\mathscr M)^*$]\label{thm:dual_of_pullback_I}
 With the same notation and assumptions as above, for any $L\in (\varphi^*\mathscr M)^*$ there is $(L_n)\subset \varphi^*\mathscr M^*$ such that ${\sf I}_\varphi(L_n)(V)\to L(V)$ in $L^1(\mu_\Y)$ for every $V\in \varphi^*\mathscr M$. 
 \end{theorem}
\begin{proof} 
Fix $L\in(\varphi^*\mathscr M)^*$ and for $E\in\Sigma_\Y$ define $L_E\in\mathscr M^*$ as $L_E(v):={\rm Pr}(\1_E^{\mu_\Y}L(\varphi^*v))$. It is clear that $L_E$ is linear and continuous, to prove $L^\infty(\mu_\X)$-linearity we notice that \eqref{eq:prpr} gives
\begin{equation}
\label{eq:nple}
|L_E(v)|\leq {\rm Pr}(\1_E^{\mu_\Y}|L|\,|v|\circ\varphi)=|v|\,{\rm Pr}(\1_E^{\mu_\Y}|L|)
\end{equation}
and the claim easily follows.

Now let $\mathcal P=(\mathcal P_k)$ be a sequence of partitions of $\Y$ as in Theorem \ref{thm:partition_Doob}, say $\mathcal P_k=\{E^k_j\}_{j}$, and for every $k\in\N$ define 
\[
L_k:=\sum_{j:\,\mu_\Y(E^k_j)>0}\1_{E^k_j}^{\mu_\Y}\frac{\varphi^*L_{E^k_j}}{{\rm Pr}(\1_{E^k_j}^{\mu_\Y})\circ\varphi}.
\]
Notice that in principle we only know that $L_k$ is an element of the $L^0$-completion of $\varphi^*\mathscr M^*$. We prove that actually $L_k\in \varphi^*\mathscr M^*$ with a uniform bound on the norm and to this aim we are going to use the inequality
\begin{equation}
\label{eq:dopop}
|{\rm Pr}(\1_E^{\mu_\Y}f)|^p\leq {\rm Pr}(\1_E^{\mu_\Y}|f|^p)\,{\rm Pr}(\1_E^{\mu_\Y})^{p-1}\qquad\forall p\in[1,\infty),\ f\in L^p(\mu_\Y),\ E\in\Sigma_\Y 
\end{equation}
that we shall prove later on. We have
\[
\begin{split}
|L_k|^p&\leq \sum_{j}\1_{E^k_j}^{\mu_\Y}\frac{|L_{E^k_j}|^p}{{\rm Pr}(\1_{E^k_j}^{\mu_\Y})^p}\circ\varphi\stackrel{\eqref{eq:nple}}\leq \sum_{j}\1_{E^k_j}^{\mu_\Y}\frac{{\rm Pr}(\1_{E^k_j}^{\mu_\Y}|L|)^p}{{\rm Pr}(\1_{E^k_j}^{\mu_\Y})^p}\circ\varphi\stackrel{\eqref{eq:dopop}}\leq \sum_{j}\1_{E^k_j}^{\mu_\Y}\frac{{\rm Pr}(\1_{E^k_j}^{\mu_\Y}|L|^p)}{{\rm Pr}(\1_{E^k_j}^{\mu_\Y})}\circ\varphi,
\end{split}
\]
and by integration we conclude that $\int |L_k|^p\,\d\mu_\Y\leq \int |L|^p\,\d\mu_\Y<\infty$.

Thus our claim is proved and hence, by $L^\infty(\mu_\Y)$-linearity and the fact that $\{\varphi^*v:v\in\mathscr M\}$ generates $\varphi^*\mathscr M$,  to prove that the $L_k$'s satisfy the conclusion it is sufficient to check that for any $v\in\mathscr M$ we have ${\sf I}_\varphi(L_k)(\varphi^*v)\to L(\varphi^*v)$ in $L^1(\mu_\Y)$.

Since for any  $v\in\mathscr M$ we have
\[
{\sf I}_\varphi(L_k)(\varphi^*v)=\sum_j\1_{E_{k,j}}^{\mu_\Y}\frac{{\rm Pr}(\1_{E^k_j}^{\mu_\Y}L(\varphi^*v))}{{\rm Pr}(\1_{E^k_j}^{\mu_\Y})}\circ\varphi,
\]
the conclusion will follow if we show that for any $f\in L^1(\mu_\Y)$ we have $f_k\to f$ in $L^1(\mu_\Y)$, where 
\[
f_k:=\sum_j\1_{E^k_j}^{\mu_\Y}\frac{{\rm Pr}(\1_{E^k_j}^{\mu_\Y}f)}{{\rm Pr}(\1_{E^k_j}^{\mu_\Y})}\circ\varphi.
\] To see this notice that the maps sending $f$ to $f_k$ are linear and with norm 1 (seen as operators from $L^1(\mu_\Y)$ into itself), thus the claim will follow if we show that the desired convergence is in place for a dense set of $f$'s. As dense set we pick the space of linear combinations of functions of the kind $\1_{E^{k}_j}^{\mu_\Y}$ for some $k,j$ (density follows from the separability assumption and the construction of partitions): for $f$ of this form  we trivially have $f_{k'}=f$ for all sufficiently big $k'$, whence the conclusion follows.

It remains to prove \eqref{eq:dopop}. If $\ell:\R\to\R$ is affine we have $\ell\circ \frac{{\rm Pr}(\1_E^{\mu_\Y}f)}{{\rm Pr}(\1_E^{\mu_\Y})}=\frac{{\rm Pr}(\1_E^{\mu_\Y}(\ell\circ f))}{{\rm Pr}(\1_E^{\mu_\Y})}$, hence by the monotonicity of ${\rm Pr}$ and recalling that any $u:\R\to\R$ convex and lower semicontinuous is the supremum of affine functions $\leq u$ we have
\[
u\circ  \frac{{\rm Pr}(\1_E^{\mu_\Y}f)}{{\rm Pr}(\1_E^{\mu_\Y})}=\sup_{\ell\leq u}\ell\circ \frac{{\rm Pr}(\1_E^{\mu_\Y}f)}{{\rm Pr}(\1_E^{\mu_\Y})}=\sup_{\ell\leq u}\frac{{\rm Pr}(\1_E^{\mu_\Y}(\ell\circ f))}{{\rm Pr}(\1_E^{\mu_\Y})}\leq \frac{{\rm Pr}(\1_E^{\mu_\Y}(u\circ f))}{{\rm Pr}(\1_E^{\mu_\Y})}.
\]
Picking $u(z):=|z|^p$ the claim  \eqref{eq:dopop} follows.
\end{proof}
\begin{remark}\label{re:sorpresa}{\rm In a topological vector space (and in fact in any uniformity) it makes sense to say that a net $(v_{\alpha})_{\alpha\in D}$ is Cauchy: it is so if for every neighbourhood $U$ of 0 there is $\alpha\in D$ such that $v_{\beta}-v_{\beta'}\in U$ for every $\beta,\beta'>\alpha$. A TVS is called complete if any Cauchy net converges and it can be proved that any Hausdorff TVS admits a unique (up to unique isomorphism) Hausdorff completion (see e.g.\ \cite{narici2010topological}).

The space $\varphi^*\mathscr M^*$ equipped with the weak$^*$ topology it inherits from its embedding in $(\varphi^*\mathscr M)^*$ is obviously Hausdorff (because the weak$^*$ topology is Hausdorff, here we are using the fact that the module dual and the Banach dual of a module are canonically isomorphic to speak about weak$^*$ topology on $(\varphi^*\mathscr M)^*$) and therefore from Theorem \ref{thm:dual_of_pullback_I} above one might wonder whether the completion of  $\varphi^*\mathscr M^*$ with the weak$^*$ topology is precisely $(\varphi^*\mathscr M)^*$ with the weak$^*$ topology.

This is not the case. The problem, well known to experts, is that the dual $B'$ of an infinite dimensional Banach space $B$ equipped with the weak$^*$ topology is not complete(!): its completion consists of the space of \emph{all} real valued linear functionals on $B$, including the discontinuous ones. To see why, let $L:B\to\R$ be linear, possibly discontinuous,  consider as direct set $D$ the collection of finite subsets of $B$, ordered by inclusion, and  for $\alpha:=\{v_1,\ldots,v_n\}\in D$ notice that the restriction of $L$ to the linear span of $\{v_1,\ldots,v_n\}$ is continuous (by finite dimensionality) and thus by Hahn-Banach it can be extended to an element $L_\alpha$ of $B'$. It is then clear that $(L_\alpha)_{\alpha\in D}$ is a Cauchy net in $B'$ that is converging to $L$. Notice that this construction heavily relies on the Hahn-Banach theorem, and thus in some rather strong version of Choice.
}\fr\end{remark}

\def\cprime{$'$} \def\cprime{$'$}


\begin{thebibliography}{10}

\bibitem{BG22}
{\sc C.~Brena and N.~Gigli}, {\em Local vector measures}.
\newblock Preprint, arXiv:2206.14864, 2022.

\bibitem{DMLP21}
{\sc S.~Di~Marino, D.~Lu\v{c}i\'{c}, and E.~Pasqualetto}, {\em Representation
  theorems for normed modules}.
\newblock Preprint, arXiv:2109.03509, 2021.

\bibitem{Fremlin3}
{\sc D.~Fremlin}, {\em Measure {T}heory: {M}easure algebras. Volume 3}, Measure
  Theory, Torres Fremlin, 2011.

\bibitem{Gigli14}
{\sc N.~Gigli}, {\em Nonsmooth differential geometry - an approach tailored for
  spaces with {R}icci curvature bounded from below}, Mem. Amer. Math. Soc., 251
  (2018), pp.~v+161.

\bibitem{GPS18}
{\sc N.~Gigli, E.~Pasqualetto, and E.~Soultanis}, {\em Differential of metric
  valued {S}obolev maps}, Journal of Functional Analysis, 278 (2020),
  p.~108403.

\bibitem{Guo-2011}
{\sc T.~X. Guo}, {\em Recent progress in random metric theory and its
  applications to conditional risk measures}, Science China Mathematics, 54
  (2011), pp.~633--660.

\bibitem{Gut93}
{\sc A.~E. Gutman}, {\em Banach bundles in the theory of lattice-normed spaces.
  {I}. {C}ontinuous {B}anach bundles}, Siberian Adv. Math., 3 (1993),
  pp.~1--55.

\bibitem{GutKop00}
{\sc A.~E. Gutman and A.~V. Koptev}, {\em Dual {B}anach bundles}, in
  Nonstandard analysis and vector lattices, vol.~525 of Math. Appl., Kluwer
  Acad. Publ., Dordrecht, 2000, pp.~105--159.

\bibitem{HLR91}
{\sc R.~Haydon, M.~Levy, and Y.~Raynaud}, {\em Randomly normed spaces}, vol.~41
  of Travaux en Cours [Works in Progress], Hermann, Paris, 1991.

\bibitem{LP18}
{\sc D.~Lu\v{c}i\'{c} and E.~Pasqualetto}, {\em The {S}erre-{S}wan theorem for
  normed modules}, Rendiconti del Circolo Matematico di Palermo Series 2, 68
  (2019), pp.~385--404.

\bibitem{LPV22}
{\sc M.~Lu\v{c}i\'{c}, E.~Pasqualetto, and I.~Vojnovi\'{c}}, {\em On the
  reflexivity properties of {B}anach bundles and {B}anach modules}.
\newblock Preprint, arXiv:2205.11608, 2022.

\bibitem{narici2010topological}
{\sc L.~Narici and E.~Beckenstein}, {\em Topological Vector Spaces}, Chapman \&
  Hall/CRC Pure and Applied Mathematics, CRC Press, 2010.

\bibitem{Pasqualetto22}
{\sc E.~Pasqualetto}, {\em Testing the {S}obolev property with a single test
  plan}, Studia Mathematica, 264 (2022), pp.~149--179.

\end{thebibliography}
\end{document}